\documentclass[11pt,letterpaper]{amsart}

\usepackage{amssymb,amsthm,mathrsfs,mathtools}
\usepackage{tikz-cd}
\usepackage{paralist}
\usepackage{aliascnt}

\theoremstyle{plain}
\newtheorem{theorem}{Theorem}[section]

\newaliascnt{lemma}{theorem}
\newtheorem{lemma}[lemma]{Lemma}
\aliascntresetthe{lemma}

\newaliascnt{corollary}{theorem}
\newtheorem{corollary}[corollary]{Corollary}
\aliascntresetthe{corollary}

\newaliascnt{proposition}{theorem}
\newtheorem{proposition}[proposition]{Proposition}
\aliascntresetthe{proposition}

\theoremstyle{definition}
\newaliascnt{definition}{theorem}
\newtheorem{definition}[definition]{Definition}
\aliascntresetthe{definition}

\newaliascnt{assumption}{theorem}
\newtheorem{assumption}[assumption]{Assumption}
\aliascntresetthe{assumption}

\theoremstyle{remark}
\newaliascnt{remark}{theorem}
\newtheorem{remark}[remark]{Remark}
\aliascntresetthe{remark}

\numberwithin{equation}{section}

\usepackage[
  plainpages=false,
  pdfpagelabels,
  pdftitle={Geometric monodromy of double covers of projective space branched along hyperplane arrangements},
  pdfauthor={Xiaopeng Xia and Jinxing Xu},
  pdfsubject={Geometric monodromy of branched double covers},
   pdfkeywords={geometric monodromy, mod-ell monodromy, ell-adic monodromy, double covers, hyperplane arrangements, Picard-Lefschetz transformations, spinor norm},
  colorlinks=true,
  linkcolor=blue,
  citecolor=blue,
  urlcolor=blue,
  hypertexnames=false
]{hyperref}
\usepackage{cleveref}

\crefname{theorem}{theorem}{theorems}
\Crefname{theorem}{Theorem}{Theorems}
\crefname{lemma}{lemma}{lemmas}
\Crefname{lemma}{Lemma}{Lemmas}
\crefname{corollary}{corollary}{corollaries}
\Crefname{corollary}{Corollary}{Corollaries}
\crefname{proposition}{proposition}{propositions}
\Crefname{proposition}{Proposition}{Propositions}
\crefname{definition}{definition}{definitions}
\Crefname{definition}{Definition}{Definitions}
\crefname{remark}{remark}{remarks}
\Crefname{remark}{Remark}{Remarks}
\crefname{assumption}{assumption}{assumptions}
\Crefname{assumption}{Assumption}{Assumptions}
\crefname{equation}{equation}{equations}
\Crefname{equation}{Equation}{Equations}

\title[Geometric monodromy of double covers]
{Geometric monodromy of double covers of $\mathbb P^n$
branched along hyperplane arrangements}

\author[]{Xiaopeng Xia}
\address{Beijing Institute of Mathematical Sciences and Applications\\ No. 544, Hefangkou Village, Huaibei Town, Huairou District, Beijing 101408, P. R. China}
\email{xpxia@mail.ustc.edu.cn}
\author[]{Jinxing Xu}
\address{School of Mathematical Sciences, University of Science and Technology of China\\ No. 96 Jinzhai Road, Hefei, Anhui 230026, P. R. China}
\email{xujx02@ustc.edu.cn}

\date{}

\begin{document}

\begin{abstract}
Let $n\ge1$, let $m\ge n+3$ be even, and let $\ell$ be an odd prime.
Over an algebraically closed field in which $2\ell$ is invertible,
we determine the geometric mod-$\ell$ and integral $\ell$-adic
monodromy of double covers of $\mathbb P^n$ branched along ordered
arrangements of $m$ hyperplanes in general position, on the negative
eigenspace of their middle cohomology. The mod-$\ell$ image is the
full symplectic group for odd $n$ and an index-two orthogonal
subgroup determined by the spinor norm and determinant for even $n$.
The integral image is the full inverse image of the finite image.
Over finite fields, we prove generic irreducibility of the numerator of
the zeta function in odd dimension and an equidistribution theorem for
Frobenius fixed spaces. For $n=1$, this recovers the
$\ell$-torsion part of the geometric Cohen--Lenstra distribution
for quadratic function fields.
\end{abstract}
  \maketitle

\section{Introduction}

The mod-\(\ell\) monodromy of a family of varieties controls both the
variation of its middle cohomology and arithmetic statistics of its fibers
over finite fields.  For hyperelliptic curves, Yu proved that the geometric
mod-\(\ell\) monodromy is the full symplectic group for every odd prime
\(\ell\) \cite{yu1997toward}; see also
\cite{achter2007integral,10.1215/S0012-7094-08-14115-8}.  This result has
applications to Cohen--Lenstra heuristics for quadratic function fields
\cite{achter2006distribution} and to the distribution of zeta functions
\cite{chavdarov1997generic,kowalski2006large}.

We prove an exact higher-dimensional analogue by studying double covers of
projective space branched along hyperplane arrangements.  A hyperelliptic
curve is a double cover of \(\mathbb P^1\) branched at distinct points;
correspondingly, we consider double covers of \(\mathbb P^n\) branched along
hyperplanes in general position and determine the geometric monodromy on the
nontrivial eigenspace of their middle cohomology.

\noindent\textbf{Conventions.}
Throughout, \(n\ge1\), the integer \(m\ge n+3\) is even, \(\ell\) is an odd
prime, and the base field \(\mathbb K\) satisfies
\(2\ell\in\mathbb K^\times\); unless stated otherwise, \(\mathbb K\) is
algebraically closed.  Except in \Cref{subsec:complex-case}, all cohomology groups are \'etale, and
\(\pi_1\) denotes the \'etale fundamental group.  A base point in \(\pi_1(S,s)\) is geometric;
for an ordinary point \(u\), we write \(\bar u\) for a chosen geometric point
above it, and for an affine group scheme \(G/R\), \(G(R)\) denotes its group of
\(R\)-points.  We suppress Tate twists in Poincar\'e pairings.  In even
dimension we choose a basis of \(\mathbb F_\ell(-n/2)\); changing it rescales
the resulting orthogonal form by a square and hence does not affect the
spinor-norm statements.

Let \(\mathfrak M\) be the normalized parameter space of ordered
\(m\)-tuples of hyperplanes in \(\mathbb P^n_{\mathbb K}\) in general position,
considered up to projective equivalence, and let
\(f\colon\mathcal X\to\mathfrak M\) be its tautological double-cover family;
see \Cref{subsection:double_cover}.  Fix a geometric point
\(s\colon\operatorname{Spec}\Omega\to\mathfrak M\), put
\(X=\mathcal X_s\), and let \(\iota\) be the deck involution.  Write
\(V_\ell=H^n(X,\mathbb F_\ell)_{(1)}
 :=H^n(X,\mathbb F_\ell)^{\iota=-1}\).  Since
\(X/\langle\iota\rangle\simeq\mathbb P^n_\Omega\), pullback gives the canonical
decomposition
\[
H^n(X,\mathbb F_\ell)
 =V_\ell\oplus H^n(\mathbb P^n_\Omega,\mathbb F_\ell).
\]
In particular, \(V_\ell=H^n(X,\mathbb F_\ell)\) when \(n\) is odd.

The dimension is
\[
 d:=\dim_{\mathbb F_\ell}V_\ell=\binom{m-2}{n};
\]
see \Cref{prop:hyperelliptic-comparison}.

Although \(X\) is generally singular, the smooth Kummer cover
\(q_s\colon Y\to X\) constructed in \Cref{subsec:kummer-cover} identifies
\(H^n(X,\mathbb F_\ell)\) with
\(H^n(Y,\mathbb F_\ell)^{\overline N}\).  Transporting the Poincar\'e pairing
on \(Y\) and restricting to \(V_\ell\) gives a non-degenerate form \(Q\),
alternating for odd \(n\) and symmetric for even \(n\), preserved by geometric
monodromy.  Thus
\[
\rho\colon\pi_1(\mathfrak M,s)\longrightarrow
\operatorname{Aut}(V_\ell,Q).
\]
We write this automorphism group as \(\operatorname{Sp}(V_\ell,Q)\) for odd
\(n\) and \(\operatorname O(V_\ell,Q)\) for even \(n\).  In the latter case
put \(q(v)=Q(v,v)/2\) and normalize the spinor norm
\(\theta\colon\operatorname O(V_\ell,Q)\to
\mathbb F_\ell^\times/(\mathbb F_\ell^\times)^2=\{\pm1\}\) by
\(\theta(r_v)=[q(v)]\) for the reflection in an anisotropic vector \(v\).

\begin{theorem}\label{thm:main-monodromy}
Let \(n\ge1\), let \(m\ge n+3\) be even, let \(\ell\) be an odd prime, and
let \(\mathbb K\) be algebraically closed with
\(2\ell\in\mathbb K^\times\).  Set
\(V_\ell=H^n(X,\mathbb F_\ell)_{(1)}\) and
\(\Gamma=\rho(\pi_1(\mathfrak M,s))\).
\begin{compactenum}[\normalfont(1)]
\item\label{item:main-monodromy-odd} If \(n\) is odd, then
\(\Gamma=\operatorname{Sp}(V_\ell,Q)\); here
\(V_\ell=H^n(X,\mathbb F_\ell)\).

\item\label{item:main-monodromy-even} If \(n\) is even, then
\[
\Gamma=
\begin{cases}
\ker\theta,
 & n\equiv0\pmod4,\\[2pt]
\ker\theta,
 & n\equiv2\pmod4\text{ and }\ell\equiv1\pmod4,\\[2pt]
\ker(\theta\cdot\det),
 & n\equiv2\pmod4\text{ and }\ell\equiv3\pmod4.
\end{cases}
\]
Here the kernels are taken in \(\operatorname O(V_\ell,Q)\).
\end{compactenum}
\end{theorem}

Thus the monodromy is maximal in the symplectic case, while in the orthogonal
case the Picard--Lefschetz reflections leave only the obstruction measured by
the spinor norm and determinant.  The Calabi--Yau case \(m=2n+2\) has been
studied over \(\mathbb C\) from the viewpoints of variations of Hodge
structure and complex monodromy
\cite{gerkmann2013monodromy,sheng2015monodromy,xu2018zariski,castor2024remarks};
our result instead determines the finite geometric mod-\(\ell\) image for all
even \(m\ge n+3\).

We also determine the integral image. Put
\(L_\ell=H^n(X,\mathbb Z_\ell)_{(1)}\), and let $Q_\ell$ be the pairing
induced by the Kummer cover, with compatible choices of Tate trivializations.
By \Cref{prop:integral-lattice-and-reduction}, this is a free lattice of
rank $d$ with perfect pairing, and its reduction is $(V_\ell,Q)$. Write
$K_\ell$ for the geometric monodromy image on $L_\ell$, and let
\[
 \operatorname{red}_\ell\colon
 \operatorname{Aut}_{\mathbb Z_\ell}(L_\ell,Q_\ell)
 \longrightarrow\operatorname{Aut}(V_\ell,Q)
\]
be reduction modulo $\ell$.

\begin{theorem}[Integral geometric monodromy]
\label{thm:intro-integral-monodromy}
With the notation and assumptions above, for every odd prime $\ell$
invertible in $\mathbb K$,
\[
 K_\ell=\operatorname{red}_\ell^{-1}(\Gamma),
\]
where $\Gamma$ is the finite group in \Cref{thm:main-monodromy}. In
particular, $K_\ell=\operatorname{Sp}(L_\ell,Q_\ell)(\mathbb Z_\ell)$
when $n$ is odd. When $n$ is even, the inverse image is taken in
\(\operatorname O(L_\ell,Q_\ell)(\mathbb Z_\ell)\).
\end{theorem}

This is proved in \Cref{thm:ell-adic-monodromy}. Thus the integral image
contains the full principal congruence subgroup of level $\ell$, and the
finite group in \Cref{thm:main-monodromy} is its only reduction constraint.

Over finite fields we further prove generic irreducibility of the middle zeta
numerator in odd dimension
(\Cref{prop:generic-zeta-irreducibility}) and, under
$q\equiv1\pmod\ell$, equidistribution of Frobenius fixed-space dimensions
modulo $\ell$ according to the corresponding finite monodromy group
(\Cref{prop:frobenius-fixed-space-distribution}); in even dimension the latter
is taken along even extension degrees. For \(n=1\), this recovers the
\(\ell\)-torsion part of Achter's geometric Cohen--Lenstra distribution for
quadratic function fields \cite{achter2006distribution}.

The proof of \Cref{thm:main-monodromy} uses Hall's criterion for large finite
symplectic and orthogonal monodromy
\cite[Thm.~3.1]{10.1215/S0012-7094-08-14115-8}.  In the orthogonal case with
$\ell=3$, we instead use the reflection-group classification in
\cite[Thm.~3.7 and Table~1]{etingof2018reflection}: the dimension of the
cohomology, scalar monodromy, and an exterior-power transvection exclude its
non-orthogonal-type alternatives.  The first geometric input is a
Picard--Lefschetz formula: allowing one minimal failure of general position
produces ordinary double points on the Kummer cover and hence symplectic
transvections for odd $n$ and orthogonal reflections for even $n$.

The main difficulty is irreducibility.  Unlike the discriminant of a Lefschetz
pencil or of the hyperelliptic family, the boundary of the ordered arrangement
space has many components, indexed by $(n+1)$-tuples of hyperplanes.  Passing
to unordered arrangements makes these components equivalent.  We show that
this equivalence is already visible in geometric monodromy by comparing with
the hyperelliptic family along
\(\operatorname{Sym}^n\mathbb P^1\simeq\mathbb P^n\).  The full
hyperelliptic monodromy of
\Cref{cor:hyperelliptic-universal-monodromy} then identifies the ordered and
unordered monodromy images, so the vanishing cycles form a single orbit and
the representation is irreducible.

\Cref{sec:the_covers} constructs the double and Kummer covers;
\Cref{sec:a_picard_lefschetz_type_formula,sec:connections_with_hyperelliptic_curves}
establish the Picard--Lefschetz formula and the hyperelliptic comparison;
\Cref{sec:main_results} proves the finite and integral monodromy theorems; and
\Cref{sec:some_consequences} gives the arithmetic applications.

\noindent\textbf{Acknowledgements.} This work was partially supported by the CAS Project for Young Scientists in Basic Research (Grant No.~YSBR-032) and by the NSFC (Grant Nos.~12271495, 12341105). Part of this work was carried out while the second-named author was visiting the Institut de Math\'ematiques de Toulouse, supported by the China Scholarship Council (CSC). He would like to thank the institute for providing an excellent research environment, and Prof.~Laurent Manivel for his warm hospitality. In particular, he is grateful to Prof.~Manivel for helpful discussions on the cohomology of symmetric products of curves.

\noindent\textbf{AI disclosure.} During the revision of this manuscript, the authors used ChatGPT (OpenAI), primarily with the GPT-5.6 Sol and GPT-6 Astra models, to assist in developing and refining the proof of \Cref{thm:main-monodromy} in the case where $n$ is even and $\ell=3$, as well as the arguments determining the integral geometric $\ell$-adic monodromy image for even $n$ (\Cref{thm:ell-adic-monodromy}). ChatGPT was also used to assist in checking and revising other mathematical arguments and references, and for language editing. The authors have checked all mathematical arguments, including those developed with AI assistance, and take full responsibility for the correctness and content of the final manuscript.

\section{The covers}\label{sec:the_covers}

\subsection{Double cover}\label{subsection:double_cover}

In this subsection, we do not assume that $\mathbb K$ is algebraically
closed.  An ordered arrangement
\(\mathfrak A=(H_1,\ldots,H_m)\) in \(\mathbb P^n_{\mathbb K}\) is in
\emph{general position} if no \(n+1\) members meet at a point, equivalently
if \(\sum_iH_i\) is a simple normal crossings divisor.

The first \(n+2\) hyperplanes form an ordered projective frame in the dual
projective space.  Since \(\operatorname{PGL}_{n+1,\mathbb K}\) acts simply
transitively on such frames, every projective-equivalence class has a unique
representative with coefficient matrix
\begin{equation}\label{eq:normalized-arrangement-matrix}
B(a)=
\begin{pmatrix}
1 & 0 & \cdots & 0 & 1 & 1 & \cdots & 1 \\
0 & 1 & \cdots & 0 & 1 & a_{11} & \cdots & a_{1,m-n-2} \\
\vdots & \vdots & \ddots & 0 & \vdots & \vdots & \ddots & \vdots \\
0 & 0 & \cdots & 1 & 1 & a_{n1} & \cdots & a_{n,m-n-2}
\end{pmatrix}.
\end{equation}
Indeed, after normalizing the first \(n+2\) columns, general position makes
every coordinate of each remaining column nonzero, and there is a unique
scaling with first coordinate \(1\).  If \(B(a)=(b_{ij})\), its \(j\)-th
column defines \(H_j\) by
\begin{equation}\label{eq:hyperplane-equation}
\sum_{i=0}^n b_{ij}x_i=0.
\end{equation}
Let
\(\mathfrak M\subset\mathbb A_{\mathbb K}^{n(m-n-2)}\) be the open
subscheme on which all maximal minors of \(B(a)\) are invertible.  Thus
\(\mathfrak M\) is a normalized (or rigidified) parameter space for ordered
arrangements up to projective equivalence.

Put
\[
\widetilde{\mathfrak M}
 =\operatorname{Spec}\mathbb K[t_{ij},\Delta^{-1}]
 \subset\mathbb A_{\mathbb K}^{(n+1)m},
\qquad
\Delta=\prod_{j_0<\cdots<j_n}
\det(t_{i,j_r})_{0\le i,r\le n}.
\]
This coefficient-vector parameter space records an \((n+1)\times m\)
matrix of defining forms.  The matrix \eqref{eq:normalized-arrangement-matrix} gives a closed
immersion \(\jmath\colon\mathfrak M\hookrightarrow\widetilde{\mathfrak M}\).

Let \(\mathbb P_{\mathbb K}(m/2,1,\ldots,1)
=\operatorname{Proj}\mathbb K[x,x_0,\ldots,x_n]\), with \(\deg x=m/2\),
and define \(\widetilde{\mathcal X}\) over \(\widetilde{\mathfrak M}\) by
\(x^2=\prod_{j=1}^m\left(\sum_{i=0}^nt_{ij}x_i\right)\).
Write \(\widetilde f\colon\widetilde{\mathcal X}\to
\widetilde{\mathfrak M}\) for the projection, and let
\(f\colon\mathcal X=\widetilde{\mathcal X}
\times_{\widetilde{\mathfrak M}}\mathfrak M\to\mathfrak M\) be its pullback.
For every geometric point \(s\colon\operatorname{Spec}\Omega\to\mathfrak M\),
the fiber \(\mathcal X_s\) is the double cover of \(\mathbb P^n_\Omega\)
branched along the corresponding normalized arrangement
(cf.~\cite[\S~3.5]{esnault1992lectures}).  We call \(f\) the tautological
double-cover family.  A geometric base point of \(\mathfrak M\) and its image
under \(\jmath\) will be denoted by the same letter.

\begin{lemma}\label{lem:lisse-cohomology}
The sheaves $R^n f_*\mathbb F_\ell$ and
$R^n\widetilde f_*\mathbb F_\ell$ are finite locally constant.
\end{lemma}
The proof is given after the Kummer-cover construction below.

\subsection{Kummer cover}\label{subsec:kummer-cover}

In \(\widetilde{\mathfrak M}\times
\mathbb P_{\mathbb K}(1,\ldots,1,2,\ldots,2)\), with coordinates
\([y_1:\cdots:y_m:x_0:\cdots:x_n]\), where each \(y_j\) has weight \(1\) and
each \(x_i\) has weight \(2\),
let \(\mathcal Y\) be defined by \(y_j^2=\ell_j\) for \(1\le j\le m\),
where \(\ell_j=\sum_{i=0}^nt_{ij}x_i\).

\begin{lemma}\label{lemma:smooth}
The projection $\mathcal Y\to\widetilde{\mathfrak M}$ is proper and smooth,
and every geometric fiber is a smooth complete intersection in
$\mathbb P^{m-1}$.
\end{lemma}
\begin{proof}
The projection is projective, hence proper.  For \(1\le j\le n+1\), set
\(x_j'=\ell_j\).  The first \(n+1\) columns of \((t_{ij})\) are invertible,
so the \(x_j'\) are relative coordinates; write
\(\ell_j=\sum_{i=1}^{n+1}t'_{ij}x_i'\) for \(j>n+1\).  The equations become
\begin{equation}\label{eq:kummer-equations}
\begin{cases}
y_j^2=x_j', & 1\le j\le n+1,\\
y_j^2=\displaystyle\sum_{i=1}^{n+1}t'_{ij}x_i', & n+2\le j\le m,
\end{cases}
\end{equation}
and every \(n+1\) columns of
\begin{equation}\label{eq:kummer-coefficient-matrix}
A=
\begin{pmatrix}
1 & 0 & \cdots & 0 & t'_{1,n+2} & \cdots & t'_{1m}\\
0 & 1 & \cdots & 0 & t'_{2,n+2} & \cdots & t'_{2m}\\
\vdots & \vdots & \ddots & \vdots & \vdots & \ddots & \vdots\\
0 & 0 & \cdots & 1 & t'_{n+1,n+2} & \cdots & t'_{n+1,m}
\end{pmatrix}
\end{equation}
are linearly independent at every geometric point.

Eliminating the \(x_i'\) identifies \(\mathcal Y\) with the subscheme of
\(\widetilde{\mathfrak M}\times\mathbb P_{\mathbb K}^{m-1}\) cut out by
\(f_j=y_j^2-\sum_{i=1}^{n+1}t'_{ij}y_i^2\), \(j>n+1\).  Put
\(p=n+1\), \(r=m-n-1\), write \(A=(I_p\mid T)\), and set
\(B=(-T^{\mathsf t}\mid I_r)\).  The relative Jacobian is
\(J=\left(\frac{\partial f_{n+1+\alpha}}{\partial y_j}\right)_{\alpha,j}
 =2B\operatorname{diag}(y_1,\ldots,y_m)\),
where \(1\le\alpha\le r\) indexes the rows (equations) and \(1\le j\le m\)
the columns (variables).

Fix a geometric point \([y_1:\cdots:y_m]\) of a fiber.  The point
\([y_1^2:\cdots:y_p^2]\in\mathbb P^{p-1}\) is well defined, since otherwise
all the \(y_j\) would vanish.  If \(y_j=0\), this point lies on the
hyperplane represented by the \(j\)-th column of \(A\).  Hence at most
\(n\) of the \(y_j\) vanish and at least \(m-n=r+1\) are nonzero.  Choose
\(S\subset\{1,\ldots,m\}\) of size \(r\) with \(y_j\ne0\) for \(j\in S\).
For a matrix \(M\) with \(m\) columns and a subset \(E\) of column indices,
write \(M_E\) for the submatrix formed by the columns indexed by \(E\).
Since \(J=2B\operatorname{diag}(y_1,\ldots,y_m)\), one has
\(\det J_S=2^r(\prod_{j\in S}y_j)\det B_S\).

We observe that
\begin{equation}\label{eq:complementary-minors-A-B}
\det B_S=\pm\det A_{S^c}.
\end{equation}
Since
\(|S^c|=p\), general position gives \(\det A_{S^c}\ne0\), and therefore
\(\det J_S=\pm2^r\left(\prod_{j\in S}y_j\right)\det A_{S^c}\ne0\).
Thus the relative Jacobian has rank \(r\) everywhere.  The relative
Jacobian criterion proves smoothness of relative dimension \(n\), and the
fibers are smooth complete intersections of \(r=m-n-1\) quadrics in
\(\mathbb P^{m-1}\).
\end{proof}

We call the fibers of \(\mathcal Y\) Kummer covers.  Let
\(N=\ker((\mathbb Z/2\mathbb Z)^m\xrightarrow{\sum}\mathbb Z/2\mathbb Z)\),
acting by sign changes in the \(y_j\).  Since \(m\) is even, the diagonal
element \(\mathbf1=(1,\ldots,1)\in N\) acts as the weighted scalar \(-1\)
and hence trivially.  It is the whole kernel: on a fiber choose a point over
the complement of the branch hyperplanes, where all \(y_j\ne0\); an element
fixing this point must change all signs equally.  Thus the effective group is
\(\overline N=N/\langle\mathbf1\rangle\), of order \(2^{m-2}\).  Quotients
and cohomological invariants for \(N\) and \(\overline N\) agree.

\begin{proposition}\label{prop:kummer-quotient}
The morphism
\[
q\colon\mathcal Y\longrightarrow\widetilde{\mathcal X},\qquad
[y_1:\cdots:y_m:x_0:\cdots:x_n]
 \longmapsto[\textstyle\prod_{j=1}^m y_j:x_0:\cdots:x_n]
\]
is finite and is the categorical quotient by $N$, equivalently by
$\overline N$.  Hence $\mathcal Y/\overline N\simeq\widetilde{\mathcal X}$
over $\widetilde{\mathfrak M}$.
\end{proposition}
\begin{proof}
Put \(\Lambda=\mathbb K[t_{ij}\mid0\le i\le n,\,1\le j\le m][\Delta^{-1}]\)
and \(A_0=\Lambda[x_0,\ldots,x_n]\), with \(\deg x_i=2\).  The relative
homogeneous coordinate ring of \(\mathcal Y\) is
\[
B=A_0[y_1,\ldots,y_m]/(y_j^2-\ell_j)_{1\le j\le m},
\qquad \deg y_j=1.
\]
The monic relations make \(B\) free over \(A_0\), with basis
\(y^\varepsilon=\prod_jy_j^{\varepsilon_j}\),
\(\varepsilon\in\{0,1\}^m\).  The character of \(N\) on
\(y^\varepsilon\) is trivial exactly when all \(\varepsilon_j\) are equal:
if \(\varepsilon_i\ne\varepsilon_j\), changing the \(i\)-th and \(j\)-th
signs acts by \(-1\).  Thus, for \(w=y_1\cdots y_m\),
\begin{equation}\label{eq:graded-invariant-ring-of-Y}
B^N\simeq
A_0[w]\big/\bigl(w^2-\textstyle\prod_{j=1}^m\ell_j\bigr),
\qquad \deg w=m.
\end{equation}

All degrees in \(B^N\) are even.  Divide the grading by two and denote the
resulting ring by \(D\).  Then \eqref{eq:graded-invariant-ring-of-Y}
identifies \(D\) with
\[
C=\Lambda[u,x_0,\ldots,x_n]/
\bigl(u^2-\textstyle\prod_{j=1}^m\ell_j\bigr),
\qquad \deg u=m/2,\quad \deg x_i=1,
\]
via \(u\mapsto w\).  Consequently
\(\operatorname{Proj}_\Lambda(B^N)\simeq
\operatorname{Proj}_\Lambda D\simeq
\operatorname{Proj}_\Lambda C=\widetilde{\mathcal X}\).

It remains to verify the scheme-theoretic quotient.  The invariant elements
\(x_0,\ldots,x_n\) define an affine open cover of \(\mathcal Y\), since the
relations would otherwise force every \(y_j\) into the same homogeneous
prime.  As \(|N|\) is invertible in \(\Lambda\), the Reynolds operator makes
invariants commute with localization:
\[
\bigl((B[x_i^{-1}])_0\bigr)^N
 =\bigl((B^N)[x_i^{-1}]\bigr)_0
 =\bigl(D[x_i^{-1}]\bigr)_0.
\]
Thus each open maps to its affine categorical quotient, and these quotients
glue.  Under \(C\simeq D\) the resulting morphism is \(q\); it is finite
because \(B\) is finite over \(B^N\).
\end{proof}

\begin{proof}[Proof of \Cref{lem:lisse-cohomology}]
Let \(F\colon\mathcal Y\to\widetilde{\mathfrak M}\) be the projection.
By \Cref{lemma:smooth} and \cite[Chap.~I, Thm.~8.9]{freitag1987etale},
\(R^nF_*\mathbb F_\ell\) is finite locally constant.  For the finite quotient
\(q\colon\mathcal Y\to\widetilde{\mathcal X}\), one has
\(R^jq_*\mathbb F_\ell=0\) for \(j>0\), hence
\(R^nF_*\mathbb F_\ell\simeq R^n\widetilde f_*(q_*\mathbb F_\ell)\).
Moreover \((q_*\mathbb F_\ell)^{\overline N}\simeq\mathbb F_\ell\), as is
seen on geometric stalks, and taking \(\overline N\)-invariants is exact
because \(\ell\nmid|\overline N|\).  Therefore
\[
\bigl(R^nF_*\mathbb F_\ell\bigr)^{\overline N}
 \simeq R^n\widetilde f_*
       \bigl((q_*\mathbb F_\ell)^{\overline N}\bigr)
 \simeq R^n\widetilde f_*\mathbb F_\ell.
\]
Thus \(R^n\widetilde f_*\mathbb F_\ell\) is a direct summand of a finite
locally constant sheaf.  The same holds for \(R^nf_*\mathbb F_\ell\) by
proper base change along \(\jmath\).
\end{proof}

Choose \(s_0\in\widetilde{\mathfrak M}(\mathbb K)\), let \(s\) be its
associated geometric point, and put \(X=\widetilde{\mathcal X}_s\).

\begin{proposition}\label{prop:cohomology-away-middle}
Let $q_s\colon Y\to X$ be the Kummer quotient, let
$\pi\colon X\to\mathbb P^n_{\mathbb K}$ be the double covering, and let
$\iota$ be its deck involution.
\begin{compactenum}[\normalfont(1)]
\item For every $i$,
\[
\pi^*\colon H^i(\mathbb P^n_{\mathbb K},\mathbb F_\ell)
 \xrightarrow{\sim}H^i(X,\mathbb F_\ell)^{\iota=1},
\qquad
H^i(X,\mathbb F_\ell)^{\iota=-1}=0\quad(i\ne n).
\]
Consequently, $\pi^*$ identifies $H^i(\mathbb P^n_{\mathbb K},\mathbb F_\ell)$
with $H^i(X,\mathbb F_\ell)$ whenever $i\ne n$.
\item Via
\[
q_s^*\colon H^n(X,\mathbb F_\ell)
 \xrightarrow{\sim}H^n(Y,\mathbb F_\ell)^{\overline N},
\]
define
\begin{equation}\label{eq:pairing-on-X-from-Y}
Q_X(a,b)=Q_Y(q_s^*a,q_s^*b),
\end{equation}
where $Q_Y$ is the middle Poincar\'e duality form on $Y$.  Then $Q_X$ is
non-degenerate and $\iota$ is an isometry.  Moreover,
\begin{equation}\label{eq:orthogonal-involution-decomposition}
H^n(X,\mathbb F_\ell)
 =H^n(X,\mathbb F_\ell)^{\iota=-1}
  \mathbin{\perp}H^n(X,\mathbb F_\ell)^{\iota=1}.
\end{equation}
The restriction $Q=Q_X|_{H^n(X,\mathbb F_\ell)_{(1)}}$ is non-degenerate, alternating
for odd $n$ and symmetric for even $n$.
\end{compactenum}
The construction is compatible with base change and geometric monodromy.
In particular, the $\iota=1$ eigensheaf of
$R^n\widetilde f_*\mathbb F_\ell$ is constant, with fiber
$H^n(\mathbb P^n_{\mathbb K},\mathbb F_\ell)$.
\end{proposition}
\begin{proof}
The full sign-change group has effective quotient
\(G=(\mathbb Z/2\mathbb Z)^m/\langle\mathbf1\rangle\).  Its subgroup
\(\overline N\) has index two, and \(G/\overline N\) acts on
\(X=Y/\overline N\) as \(\iota\).  The quotient argument above gives, for
all \(i\),
\begin{equation}\label{eq:cohomology-of-X-as-invariants}
q_s^*\colon H^i(X,\mathbb F_\ell)
 \xrightarrow{\sim}H^i(Y,\mathbb F_\ell)^{\overline N}.
\end{equation}
Applying it also to \(X/\langle\iota\rangle\simeq\mathbb P^n_{\mathbb K}\)
gives \(H^i(\mathbb P^n_{\mathbb K},\mathbb F_\ell)
\simeq H^i(X,\mathbb F_\ell)^{\iota=1}\).

Suppose \(i<n\).  Weak Lefschetz identifies
\(H^i(\mathbb P^{m-1}_{\mathbb K},\mathbb F_\ell)\) with
\(H^i(Y,\mathbb F_\ell)\), equivariantly for \(G\)
(cf.~\cite[Chap.~I, Cor.~9.4]{freitag1987etale}).  Sign changes act trivially
on projective-space cohomology, so the \((-1)\)-eigenspace vanishes.  If
\(i>n\), Poincar\'e duality gives a \(G\)-invariant perfect pairing
\[
H^i(Y,\mathbb F_\ell)\times H^{2n-i}(Y,\mathbb F_\ell)
 \longrightarrow\mathbb F_\ell(-n).
\]
Its restriction to \(\overline N\)-invariants is perfect: for the averaging
projector \(e_{\overline N}=|\overline N|^{-1}\sum_{g\in\overline N}g\), an
invariant vector orthogonal to all invariants is orthogonal to every vector.
The two residual-involution eigenspaces are likewise orthogonal and perfectly
paired.  Under \eqref{eq:cohomology-of-X-as-invariants}, the negative
parts in degrees \(i\) and \(2n-i<n\) are therefore dual; the latter vanishes,
so the former does as well.  This proves~(1).

In middle degree, the same averaging argument shows that \(Q_Y\) is
non-degenerate on the \(\overline N\)-invariants, proving non-degeneracy of
\eqref{eq:pairing-on-X-from-Y}.  Since \(\iota\) is induced by \(G\), it is
an isometry.  Its two eigenspaces are orthogonal because, for
\(\iota a=a\) and \(\iota b=-b\),
\(Q_X(a,b)=Q_X(\iota a,\iota b)=-Q_X(a,b)\); hence the restrictions are
non-degenerate.  Graded commutativity gives
\(Q_X(a,b)=(-1)^nQ_X(b,a)\).

Finally, cup product, trace, averaging, and sign changes are defined in the
families, so the construction is compatible with base change and monodromy.
The quotient argument over \(\widetilde{\mathfrak M}\) identifies the
positive eigensheaf with the constant sheaf
\(H^n(\mathbb P^n_{\mathbb K},\mathbb F_\ell)
\otimes\mathbb F_{\ell,\widetilde{\mathfrak M}}\).
\end{proof}

\section{A Picard--Lefschetz type formula}\label{sec:a_picard_lefschetz_type_formula}
\subsection{Tame fundamental group}\label{subsec:tame-fundamental-group}
We use tameness with respect to the generic points of a fixed boundary.

\begin{assumption}\label{assumptions}
Let $\bar S$ be an integral normal scheme of finite type over the
algebraically closed field $\mathbb K$, and let $S\subset\bar S$ be a
dense open subscheme whose reduced complement is a divisor
$D=\bigcup_{i=1}^rD_i$. Write $\xi_i$ for the generic point of $D_i$
and fix a geometric point $s\to S$.
\end{assumption}

\begin{definition}\label{def:tame-ramification}
For a connected finite \'etale Galois cover $X\to S$, let $\bar X$ be
the normalization of $\bar S$ in $K(X)$. We say that $X\to S$ is
\emph{tamely ramified along $D$} if
$\mathcal O_{\bar S,\xi_i}\to\mathcal O_{\bar X,a}$ is tamely ramified
for every $i$ and every $a$ above $\xi_i$: the ramification index is
prime to the residue characteristic and the residue extension is
separable \cite[Def.~2.1.2]{grothendieck1971tame}. Write
$\pi_1^D(S,s)$ for the quotient of $\pi_1(S,s)$ classifying these covers.
\end{definition}

For each $i$, choose a strict henselization
$R_i=\mathcal O_{\bar S,\xi_i}^{\mathrm{sh}}$ and put
$F_i=\operatorname{Frac}(R_i)$. An \'etale path gives an inertia map
$I_i:=\operatorname{Gal}(F_i^{\mathrm{sep}}/F_i)\to\pi_1(S,s)$, well defined up to
conjugacy. Let $P_i\subset I_i$ be wild inertia. If $p$ is the
characteristic exponent of $\mathbb K$, then
\[
I_i^t:=I_i/P_i\simeq
\widehat{\mathbb Z}^{(p)}(1)(\mathbb K)
:=\varprojlim_{(e,p)=1}\mu_e(\mathbb K).
\]
Here $p=1$ and $P_i=1$ in characteristic zero. The group
$\pi_1^D(S,s)$ is obtained by killing the closed normal subgroup
generated by the images of all $P_i$, giving
\begin{equation}\label{eq:tame-fundamental-group-diagram}
\begin{tikzcd}[ampersand replacement=\&]
 I_i \arrow[r] \arrow[d,two heads] \& \pi_1(S,s) \arrow[d,two heads] \\
 I_i^t \arrow[r] \& \pi_1^D(S,s).
\end{tikzcd}
\end{equation}
Thus a continuous representation $\rho\colon\pi_1(S,s)\to G$ with
$G$ finite factors through $\pi_1^D(S,s)$ exactly when
$\rho(P_i)=1$ for every $i$.

Fix one topological generator $\alpha$ of
$\widehat{\mathbb Z}^{(p)}(1)(\mathbb K)$. Its image in $\pi_1^D(S,s)$
under the $i$-th inertia map is called a \emph{lasso} $\gamma_i$ around
$D_i$. We use this same generator throughout.

\begin{remark}
Let $b$ be a smooth point of $D_i$ in the regular locus of $\bar S$,
lying on no other component of $D$. A regular curve transverse to $D_i$
at $b$, after shrinking around $b$, realizes the same lasso conjugacy
class: the prime-to-$p$ Kummer description of tame covers near a regular
divisor identifies its tame inertia with that above
\cite[\href{https://stacks.math.columbia.edu/tag/0EYH}{Lem.~0EYH}]{stacks-project}.
Transversality gives ramification index one, and the fixed $\alpha$
therefore gives conjugate lassos
\cite[Prop.~A~I.16]{freitag1987etale}.
\end{remark}

\begin{lemma}\label{lem:tame-pi1-generated-by-lassos}
Under \Cref{assumptions}, suppose that $\bar S$ is regular and
$\pi_1(\bar S,s)=1$. Then the normal subgroup of $\pi_1^D(S,s)$
generated by the lassos $\gamma_1,\ldots,\gamma_r$ is dense.
\end{lemma}

\begin{proof}
Let $q\colon\pi_1^D(S,s)\twoheadrightarrow G$ be a finite quotient
such that $q(\gamma_i)=1$ for every $i$, and let $X\to S$ be its
connected Galois cover. Each tame inertia image is topologically
generated by the corresponding $\gamma_i$, so all generic boundary
inertia groups of this cover are trivial. The normalization $\bar X$ of
$\bar S$ in $K(X)$ is finite and is \'etale over $S$ and at every
$\xi_i$. It is therefore \'etale in codimension one, and purity of the
branch locus makes $\bar X\to\bar S$ \'etale everywhere
\cite[Exp.~X, Thm.~3.1]{SGA1}. Since $\pi_1(\bar S,s)=1$, we have
$G=1$. The asserted density follows by passing to finite quotients.
\end{proof}

\subsection{A Picard-Lefschetz type formula}
Recall that $\widetilde{\mathfrak M}\subset \mathbb{A}_\mathbb{K} ^{(n+1)m}$ is open, and its boundary
$D_0=\mathbb{A}_\mathbb{K}^{(n+1)m}\setminus \widetilde{\mathfrak M}$ has pure codimension $1$ and
$D_0=\bigcup_{1\le j_0<\dots<j_n\le m}D_{j_0,\dots,j_n}$. Here, under the coordinates $t_{ij}$ ($0\le i\le n$, $1\le j\le m$),
$D_{j_0,\dots,j_n}$ is the hypersurface defined by the determinant
$\det(t_{i,j_r})_{0\le i,r\le n}=0$.
Equivalently, a point $(t_{ij})\in\mathbb{A}_\mathbb{K}^{(n+1)m}$ corresponds to a hyperplane arrangement
$(H_1,\dots,H_m)$ with $H_j=\{\ell_j=0\}$, $\ell_j=\sum_{i=0}^n t_{ij}x_i$, and
$D_{j_0,\dots,j_n}$ parametrizes those $(H_1,\dots,H_m)$ such that
$H_{j_0}\cap\cdots\cap H_{j_n}\neq\emptyset$.

The family $\mathcal{Y} \to \widetilde{\mathfrak{M}}$ now extends naturally to a family over $\mathbb{A}_\mathbb{K}^{(n+1)m}$, still denoted by $\mathcal{Y} \to \mathbb{A}_\mathbb{K}^{(n+1)m}$. It is defined as the closed subscheme of $\mathbb{A}_\mathbb{K}^{(n+1)m} \times \mathbb{P}_{\mathbb{K}}(1,\dots,1,2,\dots,2)$ with defining ideal generated by the elements $y_j^2 - \ell_j$ for $j = 1,\dots,m$, where $\ell_j = \sum_{i=0}^n t_{ij} x_i$. The action of the group $N=\ker((\mathbb{Z}/2\mathbb{Z})^m \to \mathbb{Z}/2\mathbb{Z})$ also extends to this larger family $\mathcal{Y} \to \mathbb{A}_\mathbb{K}^{(n+1)m}$.
\begin{proposition}\label{prop:kummer-boundary-fiber}
Let $t\in \mathbb{A}_\mathbb{K}^{(n+1)m}$ be a closed point lying in the regular locus of $D_{j_0,\dots,j_n}$ but in no other irreducible component of $D_0$. Then the singular locus of the fiber $\mathcal{Y}_t$ consists of $2^{m-n-2}$ points forming a single $\overline{N}$-orbit, and each singular point is an ordinary double point.
\end{proposition}
\begin{proof}
After relabeling, assume that $I=\{1,\dots,n+1\}$ indexes the unique
vanishing maximal minor. Every subset of at most $n$ defining forms is
linearly independent, since a dependent subset could be extended to an
$(n+1)$-subset other than $I$, giving another vanishing minor. Thus we
may take $x_0=\ell_{n+2}$ and $x_i=\ell_i$ for $1\le i\le n$, with
$\ell_{n+1}=\sum_{i=1}^n a_ix_i$ and all $a_i\ne0$.
The hyperplanes indexed by $I$ meet at $p=[1:0:\cdots:0]$, which lies
on no other $H_j$. Away from $p$, at most $n$ hyperplanes meet, and they
meet transversally. The local Jacobian calculation in the proof of
\Cref{lemma:smooth} applies here as well: the vanishing branch equations
are locally $v_i^2=u_i$ for independent parameters $u_i$, while the
remaining square-root equations are \'etale. Hence $\mathcal Y_t$ is smooth at every point lying over
$\mathbb P^n_{\mathbb K}\setminus\{p\}$.

Over $p$ we use the weight-one chart $y_{n+2}=1$, with
$z_j=y_j/y_{n+2}$ and $u_i=x_i/y_{n+2}^2$.
Eliminating $u_0=1$ and $u_i=z_i^2$ for $1\le i\le n$ gives
\[
z_{n+1}^2-\sum_{i=1}^n a_i z_i^2=0,\qquad
z_j^2=b_{0j}+\sum_{i=1}^n b_{ij}z_i^2\quad(n+3\le j\le m),
\]
where $\ell_j=b_{0j}x_0+\sum_{i=1}^n b_{ij}x_i$ and $b_{0j}\ne0$.
At a point $\xi$ over $p$, we have $z_1=\cdots=z_{n+1}=0$ and
$z_j=c_j$, with $c_j^2=b_{0j}$ for $j\ge n+3$. Since $2c_j\ne0$,
the latter equations eliminate these variables in the completed local
ring, giving
\[
\widehat{\mathcal O}_{\mathcal Y_t,\xi}\simeq
\mathbb K[[z_1,\dots,z_{n+1}]]/
\left(z_{n+1}^2-\sum_{i=1}^n a_i z_i^2\right).
\]
This quadratic form is non-degenerate, since $2$ and all $a_i$ are
invertible. Hence these are all the singular points, each an ordinary
double point. The $m-n-2$ independent choices of $c_j$ give $2^{m-n-2}$ such
points; fixing the weight-one coordinate leaves no scalar identification.
Any prescribed sign changes of the $c_j$ are realized by an element
of $N$ fixing $y_{n+2}$, after changing the sign of $y_1=0$ if needed
to make the parity even. Thus $\overline N$ acts transitively.
\end{proof}

By \Cref{prop:kummer-quotient}, the tautological coefficient-vector family
$\widetilde{\mathcal X}\xrightarrow{\widetilde f}\widetilde{\mathfrak M}$
is the quotient of the smooth Kummer family by $\overline N$. Let
$\rho^{\mathrm{tot}}\colon\pi_1(\widetilde{\mathfrak M},s)
\to\operatorname{GL}(H^n(X,\mathbb F_\ell))$ be the full middle-cohomology monodromy representation. It commutes with the
deck involution $\iota$, and by \Cref{prop:cohomology-away-middle} its
restriction to the $(+1)$-eigenspace is trivial. We write
$\rho\colon\pi_1(\widetilde{\mathfrak M},s)
\to\operatorname{GL}(H^n(X,\mathbb F_\ell)_{(1)})$ for the restriction to the $(-1)$-eigenspace. For each irreducible component
$D_{j_0\dots j_n}$ of $D_0$, fix a lasso $\gamma_{j_0\dots j_n}$ in
$\pi_1^{D_0}(\widetilde{\mathfrak M},s)$.

\begin{proposition}\label{prop:picard-lefschetz}
The monodromy representation
\[
\rho\colon\pi_1(\widetilde{\mathfrak M},s)
\longrightarrow\operatorname{GL}\bigl(H^n(X,\mathbb F_\ell)_{(1)}\bigr)
\]
factors through $\pi_1^{D_0}(\widetilde{\mathfrak M},s)$. For each
irreducible component $D_{j_0\dots j_n}$ of $D_0$, there exists a
vanishing cycle $\delta_{j_0\dots j_n}\in H^n(X,\mathbb F_\ell)_{(1)}$
such that, for every $a\in H^n(X,\mathbb F_\ell)_{(1)}$,
\begin{equation}\label{eq:picard-lefschetz}
\rho(\gamma_{j_0\dots j_n})(a)=
\begin{cases}
 a-(-1)^{\frac n2}2^{m-n-2}
  Q(a,\delta_{j_0\dots j_n})\,\delta_{j_0\dots j_n},
  & n\text{ even},\\[4pt]
 a-(-1)^{\frac{n-1}{2}}2^{m-n-2}
  Q(a,\delta_{j_0\dots j_n})\,\delta_{j_0\dots j_n},
  & n\text{ odd}.
\end{cases}
\end{equation}
Here $Q=Q_X|_{H^n(X,\mathbb F_\ell)_{(1)}}$ is the form constructed in
\Cref{prop:cohomology-away-middle}. If $n$ is even, then
$Q(\delta_{j_0\dots j_n},\delta_{j_0\dots j_n})
 =(-1)^{n/2}/2^{m-n-3}$.
\end{proposition}
\begin{proof}
Let $Y=\mathcal Y_s$. By \Cref{prop:cohomology-away-middle}, pullback gives
the isometric identification $q_s^*\colon H^n(X,\mathbb F_\ell)
\xrightarrow{\sim}H^n(Y,\mathbb F_\ell)^{\overline N}$, where the total form $Q_X$ on the left is the restriction of the Poincar\'e
form $Q_Y$ on the right. The full monodromy representation
$\rho^{\mathrm{tot}}$ is likewise the restriction of the monodromy action on
$H^n(Y,\mathbb F_\ell)$ to the invariant part.

Fix an irreducible component $D_I$ of $D_0$. The local coordinate
calculation in the proof of \Cref{prop:kummer-boundary-fiber} also
applies over the strict henselization of the ambient space at the
generic point of $D_I$. Indeed, all the other maximal minors are units,
and the coefficient defining the degeneration is a uniformizer $\tau$.
After \'etale localization at a geometric singular point, the equations
reduce to
\[
z_{n+1}^2-\sum_{j=1}^n a_jz_j^2=\tau,
\qquad a_j\in R^\times,
\]
where $R$ is this strictly henselian discrete valuation ring. Thus the
total space is regular and the special fiber has ordinary quadratic
singularities. The Picard--Lefschetz formula
\cite[Exp.~XV, Thm.~3.4]{deligne2006groupes} shows that wild inertia
acts trivially on $H^n(Y,\mathbb F_\ell)$. Applying this at every
boundary component and using \Cref{subsec:tame-fundamental-group}, we
obtain the factorization of both $\rho^{\mathrm{tot}}$ and $\rho$
through $\pi_1^{D_0}(\widetilde{\mathfrak M},s)$.

Suppose that $\gamma_{j_0\dots j_n}$ is represented by a regular curve
meeting $D_{j_0\dots j_n}$ transversally at a point $t$ lying on no other
component of $D_0$. Put $r=2^{m-n-2}$, and let $p_1,\dots,p_r$ be the singular points of $\mathcal Y_t$. For each
$i$, let $\delta_i$ be the vanishing cycle associated with $p_i$. The
stabilizer in $\overline N$ of any $p_i$ changes the signs of an even number
of the local coordinates $z_1,\dots,z_{n+1}$ in the normalized chart above.
It therefore has determinant
$1$ on the quadratic local model and acts trivially on the one-dimensional
vanishing-cycle space. After choosing one sign, the cycles
$\delta_1,\dots,\delta_r$ may consequently be chosen to form a single
$\overline N$-orbit. Thus, for every
$a\in H^n(Y,\mathbb F_\ell)^{\overline N}$, one has
$Q_Y(a,\delta_1)=\cdots=Q_Y(a,\delta_r)$. Set
\[
\delta_{j_0\dots j_n}
 \coloneqq\frac1r\sum_{i=1}^r\delta_i
 \in H^n(Y,\mathbb F_\ell)^{\overline N}
 \simeq H^n(X,\mathbb F_\ell).
\]
Restricting the Picard--Lefschetz formula on $Y$ to the invariant part gives,
for every $a\in H^n(X,\mathbb F_\ell)$,
\begin{equation}\label{eq:picard-lefschetz-total}
\rho^{\mathrm{tot}}(\gamma_{j_0\dots j_n})(a)=
\begin{cases}
a-(-1)^{\frac n2}r
 Q_X(a,\delta_{j_0\dots j_n})\,\delta_{j_0\dots j_n},
  & n\text{ even},\\[4pt]
a-(-1)^{\frac{n-1}{2}}r
 Q_X(a,\delta_{j_0\dots j_n})\,\delta_{j_0\dots j_n},
  & n\text{ odd}.
\end{cases}
\end{equation}

We now locate the averaged vanishing cycle in the involution decomposition:
the remaining point to prove is the membership
\[
\delta_{j_0\dots j_n}\in H^n(X,\mathbb F_\ell)^{\iota=-1}
=H^n(X,\mathbb F_\ell)_{(1)},
\]
which will allow us to restrict \eqref{eq:picard-lefschetz-total} to the
$(-1)$-eigenspace and obtain \eqref{eq:picard-lefschetz}.
The full monodromy commutes with $\iota$ and is the identity on the
$(+1)$-eigenspace. Consequently,
\begin{equation}\label{eq:image-PL-in-negative-eigenspace}
\operatorname{im}\bigl(
\rho^{\mathrm{tot}}(\gamma_{j_0\dots j_n})-\operatorname{id}
\bigr)
 \subset H^n(X,\mathbb F_\ell)^{\iota=-1}.
\end{equation}
If $\delta_{j_0\dots j_n}=0$, then
$\delta_{j_0\dots j_n}\in H^n(X,\mathbb F_\ell)^{\iota=-1}$ holds
trivially. Otherwise,
the non-degeneracy of $Q_X$ gives an $a\in H^n(X,\mathbb F_\ell)$ such that
$Q_X(a,\delta_{j_0\dots j_n})\ne0$. Since $r$ is invertible in
$\mathbb F_\ell$, \eqref{eq:picard-lefschetz-total}
shows that $\delta_{j_0\dots j_n}$ is a nonzero scalar multiple of
\(\bigl(\rho^{\mathrm{tot}}(\gamma_{j_0\dots j_n})
-\operatorname{id}\bigr)(a)\). By
\eqref{eq:image-PL-in-negative-eigenspace}, the latter vector lies in
$H^n(X,\mathbb F_\ell)^{\iota=-1}$, and hence
\(\delta_{j_0\dots j_n}\in H^n(X,\mathbb F_\ell)^{\iota=-1}\) as well.
Restricting \eqref{eq:picard-lefschetz-total} to this
eigenspace gives \eqref{eq:picard-lefschetz}.

Finally, suppose that $n$ is even. The local vanishing cycles attached to
distinct ordinary double points are mutually orthogonal and satisfy
$Q_Y(\delta_i,\delta_i)=(-1)^{n/2}2$. Therefore
\[
Q_X(\delta_{j_0\dots j_n},\delta_{j_0\dots j_n})
 =\frac{(-1)^{n/2}2}{r}
 =\frac{(-1)^{n/2}}{2^{m-n-3}},
\]
as asserted.
\end{proof}

Now embed $\mathbb A_{\mathbb K}^{(n+1)m}$ into
$\mathbb P_{\mathbb K}^{(n+1)m}$, and let
$H_\infty=\mathbb P_{\mathbb K}^{(n+1)m}
\setminus\mathbb A_{\mathbb K}^{(n+1)m}$ be the hyperplane at infinity. Then we have $\widetilde{\mathfrak{M}} \subset \mathbb{P}_{\mathbb{K}}^{(n+1)m}$, and the boundary divisor is
\( D \coloneqq D_0 \cup H_{\infty}\).

\begin{proposition}\label{prop:unramified-at-infinity}
The monodromy representation
\[
\rho\colon\pi_1(\widetilde{\mathfrak{M}},s)
\longrightarrow\operatorname{GL}(H^n(X,\mathbb{F}_\ell)_{(1)})
\]
factors through the tame fundamental group \(\pi_1^{D}(\widetilde{\mathfrak{M}},s)\). Moreover, for any lasso \(\gamma_\infty\) around the divisor \(H_\infty\), one has \(\rho(\gamma_\infty)=\operatorname{id}\).
\end{proposition}
\begin{proof}
By \Cref{prop:picard-lefschetz}, the monodromy representation is already tamely ramified along $D_0$. Hence it remains to analyze the monodromy around $H_{\infty}$.
Assume that the homogeneous coordinates on $\mathbb{P}_{\mathbb{K}}^{(n+1)m}$ are
$[t_0 : t_{ij} \mid 0 \le i \le n,\ 1 \le j \le m]$, so that the hyperplane at
infinity is given by $H_\infty = \{t_0 = 0\}$. In these homogeneous coordinates,
$\widetilde{\mathcal{X}}$ is defined in
$\mathbb{P}_{\mathbb{K}}(\frac{m}{2},1,\dots,1)\times \widetilde{\mathfrak{M}}$
by the equation
\begin{equation}\label{eq:homogenized-double-cover}
  t_0^mx^{2}-\prod_{j=1}^m\Bigl(\sum_{i=0}^{n} t_{ij} x_i\Bigr)=0.
\end{equation}
Let $U$ be a Zariski open subset of $\mathbb{P}_{\mathbb{K}}^{(n+1)m}\setminus\overline{D_0}$
such that $U\cap H_{\infty}\neq \varnothing$ and $t_{01}$ is nowhere vanishing on $U$, where $\overline{D_0}$ denotes the Zariski closure of $D_0$ in $\mathbb{P}_{\mathbb{K}}^{(n+1)m}$.
Then on $\widetilde{\mathfrak{M}}\cap U$, the defining equation \eqref{eq:homogenized-double-cover}
takes the form
\begin{equation}\notag
  \left(\frac{t_0}{t_{01}}\right)^{m} x^{2}
  - \prod_{j=1}^{m}
    \left(\sum_{i=0}^{n} \frac{t_{ij}}{t_{01}}\, x_i\right)
  = 0.
\end{equation}
Define $\widetilde{\mathcal{X}}'$ to be the closed subscheme of
$\mathbb{P}_{\mathbb{K}}\!\left(\frac{m}{2},1,\dots,1\right)\times U$
cut out by the equation
\(x^{2}-\prod_{j=1}^m\Bigl(\sum_{i=0}^{n} \frac{t_{ij}}{t_{01}} x_i\Bigr)=0\).
Let $f'\colon\widetilde{\mathcal{X}}'\longrightarrow U$ be the morphism
induced by the projection onto the second factor.
The morphism
\[
\psi\colon U\longrightarrow\widetilde{\mathfrak M},\qquad
[t_0:t_{ij}]\longmapsto(t_{ij}/t_{01}),
\]
is well defined, since $t_{01}$ and all the relevant maximal
minors are invertible on $U$. The family $f'$ is the pullback
of $\tilde f$ along $\psi$. Hence proper base change and
\Cref{lem:lisse-cohomology} give
\[
R^nf'_*\mathbb F_\ell\simeq
\psi^*(R^n\tilde f_*\mathbb F_\ell),
\]
which is finite locally constant on $U$.
Moreover, over $\widetilde{\mathfrak{M}}\cap U$, the restriction $\widetilde{\mathcal{X}}'|_{\widetilde{\mathfrak{M}}\cap U}=f'^{-1}(\widetilde{\mathfrak{M}}\cap U)$ is isomorphic to $\widetilde{\mathcal{X}}|_{\widetilde{\mathfrak{M}}\cap U}=\tilde{f}^{-1}(\widetilde{\mathfrak{M}}\cap U)$ via the coordinate change $x\mapsto  (\frac{t_0}{t_{01}})^{\frac{m}{2}}x$. Hence the local system $(R^n\tilde{f}_* \mathbb{F}_\ell)|_{\widetilde{\mathfrak{M}}\cap U}$ extends to the local system \(R^n f'_* \mathbb{F}_\ell\) over $U$.
As $H_{\infty}\cap U\neq \emptyset$, we see the monodromy representation \(\rho\) is unramified along
\(H_\infty\); equivalently, it is tamely ramified along \(H_\infty\), and for any
lasso \(\gamma_\infty\) around \(H_\infty\) we have
\(\rho(\gamma_\infty)=\operatorname{id}\).
\end{proof}

Write the chosen point as
$s_0=(v_1,\ldots,v_m)\in\widetilde{\mathfrak M}(\mathbb K)$, with
$v_i\in\mathbb K^{n+1}$. For each \(i\), define
\[
\begin{aligned}
\varphi_i\colon\mathbb G_{m,\mathbb K}&\longrightarrow
\widetilde{\mathfrak M},\\
t&\longmapsto(v_1,\ldots,v_{i-1},tv_i,v_{i+1},\ldots,v_m).
\end{aligned}
\]
The squaring cover $[2]\colon\mathbb G_{m,\mathbb K}
\to\mathbb G_{m,\mathbb K}$, $t\mapsto t^2$, is finite \'etale Galois with deck group scheme \(\mu_2\); on geometric
points its deck group is \(\{\pm1\}\). It gives an exact sequence
\[
1\longrightarrow\pi_1(\mathbb G_{m,\mathbb K},1)
 \xrightarrow{\ [2]_*\ }
\pi_1(\mathbb G_{m,\mathbb K},1)
 \longrightarrow\{\pm1\}\longrightarrow1.
\]
Choose \(\eta\in\pi_1(\mathbb G_{m,\mathbb K},1)\) mapping to \(-1\).

\begin{proposition}\label{prop:coefficient-scaling-monodromy}
The representation
\[
\rho\circ\varphi_{i*}\colon
\pi_1(\mathbb G_{m,\mathbb K},1)
\longrightarrow
\operatorname{GL}\bigl(H^n(X,\mathbb F_\ell)_{(1)}\bigr)
\]
is trivial on \(\operatorname{im}([2]_*)\), and
\[
\rho(\varphi_{i*}\eta)=-\operatorname{id}.
\]
\end{proposition}

\begin{proof}
Write \(v_j=(a_{j0},\ldots,a_{jn})\) and
\(\ell_j=\sum_{r=0}^na_{jr}x_r\).  Under
\(\varphi_i\circ[2]\), the pullback of \(\widetilde{\mathcal X}\) is
defined by \(x^2-t^2\prod_{j=1}^m\ell_j=0\). The coordinate change \(x\mapsto x/t\) identifies it with the constant
family \(X\times\mathbb G_{m,\mathbb K}\).  Hence
\(\rho\circ\varphi_{i*}\circ[2]_*\) is trivial.  Under this
trivialization, the deck transformation \(t\mapsto-t\) acts on \(X\) by
\(x\mapsto-x\), namely by the double-cover involution.  It therefore acts
as \(-\operatorname{id}\) on
\(H^n(X,\mathbb F_\ell)_{(1)}\), proving the second assertion.
\end{proof}

For $\gamma\in \pi_1(\widetilde{\mathfrak{M}}, s)$ and $a\in  H^n(X,\mathbb{F}_\ell)_{(1)}$,
we also write \(\gamma\cdot a\) for \(\rho(\gamma)(a)\).
Let \(S\) be the following set of vanishing cycles
\[
S \coloneqq \{\;\gamma\cdot \delta_{j_0\dots j_n} \mid \gamma\in \pi_1(\widetilde{\mathfrak{M}}, s),  1\le j_0<\dots < j_n\le m\;\}.
\]
For any $\delta\in S$, define the following Picard--Lefschetz transformation:
\[T_\delta(a) \coloneqq
\begin{cases}
 a-(-1)^{\frac{n}{2}}2^{m-n-2}Q(a,\delta)\delta, & n \text{ even},\\[4pt]
 a-(-1)^{\frac{n-1}{2}}2^{m-n-2}Q(a,\delta)\delta, & n \text{ odd}.
\end{cases}
  \]
\begin{proposition}\label{prop:vanishing-cycles-generate}
$\rho(\pi_1(\widetilde{\mathfrak{M}}, s))$ is generated by $\bigl\{T_\delta\mid \delta\in S\bigr\}$, and the set \(S\) linearly spans the space \(H^n(X,\mathbb{F}_\ell)_{(1)}\).
\end{proposition}

\begin{proof}
Since $\pi_1(\mathbb{P}_{\mathbb{K}}^{(n+1)m}, s)=1$, it follows from
\Cref{lem:tame-pi1-generated-by-lassos} that the boundary lassos and
their conjugates generate a dense subgroup of
$\pi_1^D(\widetilde{\mathfrak{M}},s)$. By
\Cref{prop:unramified-at-infinity}, the image of the lasso at infinity
is trivial, so $\operatorname{im}\rho$ is generated by the conjugates of
the $\rho(\gamma_{j_0\dots j_n})$.
For $g\in\pi_1(\widetilde{\mathfrak M},s)$, invariance of $Q$ and
\eqref{eq:picard-lefschetz} give
\[
 \rho(g)\rho(\gamma_{j_0\dots j_n})\rho(g)^{-1}
 =T_{g\cdot\delta_{j_0\dots j_n}}.
\]
Thus the transformations $T_\delta$, $\delta\in S$, generate
$\operatorname{im}\rho$.

Now take any $a\in H^n(X,\mathbb{F}_\ell)_{(1)}$ satisfying
$Q(a,\delta)=0$ for all $\delta\in S$.
Then $a$ is fixed by every $T_{\delta}$, and hence by $\operatorname{im}\rho$.
On the other hand, by \Cref{prop:coefficient-scaling-monodromy}, we have
$-\operatorname{id}\in \operatorname{im}\rho$; therefore $a=-a$, and thus $a=0$.
Since $Q$ is non-degenerate on $H^n(X,\mathbb{F}_\ell)_{(1)}$, we conclude that the linear
span of $S$ is all of $H^n(X,\mathbb{F}_\ell)_{(1)}$.
\end{proof}

\subsection{The hyperelliptic monodromy input}
We use the following configuration-space consequence of the
hyperelliptic monodromy theorem of Yu, in the form proved by
Achter and Pries \cite[Thm.~3.4]{achter2007integral}.

\begin{theorem}[Yu]\label{thm:yu-one-parameter}
Let $\mathbb K$ be an algebraically closed field in which
$2\ell$ is invertible, where $\ell$ is an odd prime. Let $m\ge4$ be
even, and put
\[
U_m=\operatorname{Spec}\mathbb K\left[
\lambda_1,\ldots,\lambda_{m-1},
\left(\prod_{i<j}(\lambda_i-\lambda_j)\right)^{-1}
\right].
\]
Let $h\colon\mathcal C_U\to U_m$ be the smooth proper hyperelliptic
family with affine equation
\[
y^2=\prod_{i=1}^{m-1}(x-\lambda_i).
\]
Then, for a geometric point $s$ of $U_m$, the representation
\[
\pi_1(U_m,s)\longrightarrow
\operatorname{Sp}\bigl(H^1(\mathcal C_{U,s},\mathbb F_\ell)\bigr)
\simeq\operatorname{Sp}_{m-2}(\mathbb F_\ell)
\]
is surjective.
\end{theorem}

\begin{proof}
Put $g=(m-2)/2$, and let $\mathscr H_g^{\mathrm{lab}}$ denote the
moduli stack of smooth hyperelliptic curves of genus $g$, equipped
with their hyperelliptic involution and with all ramification points
labeled. It is the smooth-curve locus of the labeled admissible-cover
stack considered in \cite[Sects.~2.2 and 2.4]{achter2007integral}.
By \cite[Sect.~3.1 and Thm.~3.4]{achter2007integral}, its tautological
family has full symplectic mod-$\ell$ monodromy. Indeed, the locus with abelian relative Jacobian is normal, and
the smooth-curve locus is a dense open substack of it. Restriction therefore
induces a surjection on fundamental groups and does not change
this monodromy image.

Let $G=\mathbb G_m\ltimes\mathbb G_a$, where $\mathbb G_m$ acts
on $\mathbb G_a$ by $a\cdot b=a^2b$. Its action on $U_m$ is
\[
(a,b)\cdot(\lambda_1,\ldots,\lambda_{m-1})
  =(a^2\lambda_1+b,\ldots,a^2\lambda_{m-1}+b),
\]
and it lifts to the family by
\[
(a,b)\colon(x,y)\longmapsto(a^2x+b,a^{m-1}y).
\]
Placing the last labeled branch point at infinity and making the
odd-degree defining polynomial monic identifies
\[
\mathscr H_g^{\mathrm{lab}}\simeq[U_m/G].
\]
Here these coordinate normalizations are made locally on the base.
The element $(-1,0)$ fixes the branch configuration and acts on the
curve by the hyperelliptic involution; thus the quotient retains this
automorphism. Since $G$ is smooth and connected, the $G$-torsor
$U_m\to[U_m/G]$ induces a surjection on fundamental groups.
The pullback of the tautological family therefore has full
symplectic mod-$\ell$ monodromy. Finally, passing from the
$\ell$-torsion of the Jacobian to its dual representation on $H^1$
preserves the full symplectic image.
\end{proof}

\begin{corollary}\label{cor:hyperelliptic-universal-monodromy}
Let \(f_1\colon\widetilde{\mathcal C}
\to\widetilde{\mathfrak M}_1\) be the tautological hyperelliptic family
over the ordered coefficient-vector parameter space introduced in
\Cref{sec:connections_with_hyperelliptic_curves}. Then the
geometric mod-\(\ell\) monodromy representation
\[
\pi_1(\widetilde{\mathfrak{M}}_1,s)\longrightarrow
\operatorname{Sp}\bigl(H^1(C,\mathbb{F}_\ell)\bigr)
\]
is surjective.
\end{corollary}

\begin{proof}
Define
\[
\begin{aligned}
j\colon U_m&\longrightarrow\widetilde{\mathfrak M}_1,\\
(\lambda_1,\ldots,\lambda_{m-1})
&\longmapsto
\bigl((1,-\lambda_1),\ldots,(1,-\lambda_{m-1}),(0,1)\bigr).
\end{aligned}
\]
The vectors are nonzero and pairwise nonproportional on $U_m$, so $j$
lands in the ordered coefficient-vector parameter space. The pullback
of the tautological family has homogeneous equation
\[
w^2=x_1\prod_{i=1}^{m-1}(x_0-\lambda_i x_1),
\]
which on the chart $x_1=1$ becomes
\[
y^2=\prod_{i=1}^{m-1}(x-\lambda_i).
\]
Thus the pullback is the family in \Cref{thm:yu-one-parameter}, whose
monodromy is the full symplectic group. Its image is contained in the
monodromy image of $f_1$, and the latter is itself contained in the same
symplectic group. Hence the monodromy of $f_1$ is surjective.
\end{proof}

\section{Connections with hyperelliptic curves}\label{sec:connections_with_hyperelliptic_curves}
There is a natural family of hyperplane arrangements, parameterized by
configurations of points on $\mathbb P^1_{\mathbb K}$, for which the
corresponding double cover of $\mathbb P^n_{\mathbb K}$ is obtained from a
hyperelliptic curve. We describe this construction and compare the relevant
cohomology groups.

Identify $\operatorname{Sym}^n\mathbb P^1_{\mathbb K}$ with
$\mathbb P^n_{\mathbb K}$. For pairwise distinct points
$p_1,\dots,p_m\in\mathbb P^1_{\mathbb K}$, let
\[
H_i=p_i+\operatorname{Sym}^{n-1}\mathbb P^1_{\mathbb K}
 =\{D\in\operatorname{Sym}^n\mathbb P^1_{\mathbb K}\mid D-p_i
 \text{ is effective}\}.
\]
Write $v_n(a,b)=(a^n,a^{n-1}b,\ldots,b^n)^{\mathsf t}$. If
$p_i=[a_i:b_i]$, then $H_i$ has coefficient vector $v_n(a_i,b_i)$. The
Vandermonde determinant shows that $(H_1,\dots,H_m)$ is in general
position. Let $C$ be the double cover of $\mathbb P^1_{\mathbb K}$ branched
along $\sum_i p_i$, and let $X_C$ be the double cover of
$\mathbb P^n_{\mathbb K}$ branched along $\sum_iH_i$. If $N'$ is the kernel
of the summation map $(\mathbb Z/2\mathbb Z)^n\to\mathbb Z/2\mathbb Z$,
then
$X_C\simeq C^n/(N'\rtimes S_n)$; indeed, both sides are double covers of
$\mathbb P^n_{\mathbb K}$ with the same branch divisor
\cite[Lem.~2.5]{sheng2015monodromy}.

Although the quotient morphism $\pi\colon C^n\to X_C$ suggests comparing
$H^n(X_C,\mathbb F_\ell)$ with
$H^n(C^n,\mathbb F_\ell)^{N'\rtimes S_n}$, this comparison does not follow
from averaging: $|N'\rtimes S_n|=2^{n-1}n!$ and $\ell$ may divide $n!$.
Instead, we realize $\bigwedge^nH^1(C,\mathbb F_\ell)$ as an alternating
subspace of $H^n(C^n,\mathbb F_\ell)$. We first work over $\mathbb C$.

\subsection{\texorpdfstring{The $\mathbb K=\mathbb C$ case}{The K=C case}}\label{subsec:complex-case}
All cohomology groups in this subsection are singular cohomology groups.
Set $Z=C^n/S_n$ and write
$C^n\xrightarrow{\pi_1}Z\xrightarrow{\pi_2}X_C$ for the quotient maps.

\begin{proposition}\label{prop:transfer-splitting}
For every $i\ge0$, the pullback
$\pi_2^*\colon H^i(X_C,\mathbb Z[\tfrac12])\to
H^i(Z,\mathbb Z[\tfrac12])$ admits a left inverse.
\end{proposition}

\begin{proof}
Put $G=N'\rtimes S_n$, $H=S_n$, and $d=[G:H]=2^{n-1}$, so that
$Z=C^n/H$ and $X_C=C^n/G$. The map $\pi_2$ is a $d$-fold ramified
covering, with multivalued inverse
\[
C^n/G\longrightarrow\operatorname{Sym}^d(C^n/H),\qquad
[x]_G\longmapsto\sum_{Hg\in H\backslash G}[gx]_H.
\]
Changing $x$ or the representatives of the left cosets only permutes the
terms. Moreover, the composite with $\operatorname{Sym}^d(\pi_2)$ is the
$d$-fold diagonal, and $[x]_H$ occurs in the multivalued inverse of
$\pi_2([x]_H)$ (the term indexed by the identity coset). This is the standard orbit-projection construction; see
\cite[Prop.~1.5]{smith1983transfer} and
\cite[\S~1]{dold1986ramified}. Since $Z$ and $X_C$ have the homotopy type
of finite CW-complexes, the cohomological transfer
\cite[Thm.~5.4]{aguilar2006transfers} gives
$\operatorname{tr}_{\pi_2}\circ\pi_2^*=d\,\operatorname{id}$. Thus
$d^{-1}\operatorname{tr}_{\pi_2}$ is the required left inverse over
$\mathbb Z[\tfrac12]$.
\end{proof}

Put $R=\mathbb Z[\tfrac12]$. Macdonald's calculation shows that
$H^*(Z,\mathbb Z)$ is torsion-free \cite[(12.3)]{macdonald1962symmetric}.
Together with \Cref{prop:transfer-splitting}, this implies that
$H^i(Z,R)$ and $H^i(X_C,R)$ are finite free $R$-modules for every $i$. Moreover, the rational pullbacks identify
$H^n(X_C,\mathbb Q)$ and $H^n(Z,\mathbb Q)$ with the
$N'\rtimes S_n$- and $S_n$-invariants in $H^n(C^n,\mathbb Q)$,
respectively \cite[(1.2)]{macdonald1962symmetric}. Hence the pullbacks are
injective over $R$, and we regard these groups as lattices in
$H^n(C^n,\mathbb Q)$.

Let $A=H^1(C,R)$, let $W\subset H^n(C^n,R)$ be the K\"unneth summand
identified with $A^{\otimes n}$, and, for the projections
$p_i\colon C^n\to C$, define
\[
\begin{aligned}
\operatorname{Alt}\colon\bigwedge^nA&\longrightarrow W,\\
\alpha_1\wedge\cdots\wedge\alpha_n&\longmapsto
\sum_{\sigma\in S_n}\operatorname{sgn}(\sigma)\,
 p_1^*\alpha_{\sigma(1)}\cup\cdots\cup p_n^*\alpha_{\sigma(n)}.
\end{aligned}
\]
Set $W_1=\operatorname{im}(\operatorname{Alt})$.

\begin{lemma}\label{lemma:alternating-lattice-direct-summand}
The map $\operatorname{Alt}$ is injective, and $W_1$ is a direct summand of
$H^n(C^n,R)$.
\end{lemma}

\begin{proof}
Choose an $R$-basis $e_1,\dots,e_d$ of $A$. For
$I=(i_1<\cdots<i_n)$, put
$a_I=\operatorname{Alt}(e_{i_1}\wedge\cdots\wedge e_{i_n})$. In the
standard tensor basis of $A^{\otimes n}$, the supports of the $a_I$ are
pairwise disjoint, so $\operatorname{Alt}$ is injective. Let $T_I$ be the
span of the $n!$ tensors obtained by permuting
$e_{i_1},\dots,e_{i_n}$, and let $T_{\mathrm{rep}}$ be the span of tensors
with a repeated index. Since the coefficient of
$e_{i_1}\otimes\cdots\otimes e_{i_n}$ in $a_I$ is $1$, this tensor may be
replaced by $a_I$ in a basis of $T_I$. Thus $T_I=Ra_I\oplus T_I'$ and
\[
A^{\otimes n}=W_1\oplus
\left(T_{\mathrm{rep}}\oplus\bigoplus_I T_I'\right).
\]
Finally, $W$ is the $(1,\dots,1)$ K\"unneth summand of $H^n(C^n,R)$.
Hence $W_1$ is a direct summand of $H^n(C^n,R)$. No division by $n!$ is
used.
\end{proof}

Let $H^n(X_C,R)_{(1)}$ denote the $(-1)$-eigenspace for the double-cover
involution.

\begin{proposition}\label{prop:middle-cohomology-lattice}
Under the pullback embeddings above,
$\pi_1^*\circ\pi_2^*$ induces an isomorphism
\[
H^n(X_C,R)_{(1)}\xrightarrow{\ \sim\ }W_1.
\]
\end{proposition}

\begin{proof}
Let $V=(\pi_1^*\circ\pi_2^*)H^n(X_C,R)_{(1)}$.
By K\"unneth, $W_1$ is the $(-1)$-eigenspace in
$H^n(C^n,R)^{N'\rtimes S_n}$, so $V\subseteq W_1$;
over $\mathbb Q$ they are equal.

Now we claim $W_1\subseteq\pi_1^*H^n(Z,R)$. Indeed, by extending
coefficients from $\mathbb Z$ to $R$, this inclusion can be obtained
from the integral description of the cohomology of symmetric
products; see \cite[Prop.~4.3.11]{seroul1972anneau}.
For the reader's convenience, we verify it directly.
Write $g$ for the genus of $C$, let $\omega\in H^2(C,\mathbb Z)$
be the class of a point, and choose a symplectic basis
$e_1,f_1,\ldots,e_g,f_g$ of $H^1(C,\mathbb Z)$, normalized by
$e_i\cup f_j=\delta_{ij}\omega$.
The Abel--Jacobi map gives classes $\xi_v\in H^1(Z,\mathbb Z)$,
for $v\in H^1(C,\mathbb Z)$, with $\pi_1^*\xi_v=\sum_{a=1}^n p_a^*v$.
Fix $p_0\in C$ and let $\eta\in H^2(Z,\mathbb Z)$ be the class
of the divisor $p_0+\operatorname{Sym}^{n-1}C$.
Since its pullback is the sum of the coordinate divisors, we have
$\pi_1^*\eta=\sum_{a=1}^n p_a^*\omega$ and
$\pi_1^*(\xi_{e_i}\xi_{f_i}-\eta)=\sum_{a\ne b}p_a^*e_i\cup p_b^*f_i$.
Up to sign, every basis wedge can be written as
\[
w=\left(\bigwedge_{r\in J}(e_r\wedge f_r)\right)
   \wedge v_1\wedge\cdots\wedge v_s,
\]
where $J\subseteq\{1,\ldots,g\}$ indexes its complete symplectic
pairs, $s=n-2|J|$, and the $v_j$ are the remaining basis vectors.
A straightforward calculation gives
\[
\operatorname{Alt}(w)
=\pi_1^*\left(
 \prod_{r\in J}(\xi_{e_r}\xi_{f_r}-\eta)
 \prod_{j=1}^{s}\xi_{v_j}
\right).
\]
Thus $W_1\subseteq\pi_1^*H^n(Z,R)$.

By \Cref{prop:transfer-splitting}, after splitting the involution
eigenspaces we may write $H^n(Z,R)=\pi_2^*H^n(X_C,R)_{(1)}\oplus M$.
Every $\alpha\in W_1$ can therefore be written as
$\alpha=\pi_1^*(\pi_2^*\beta)+\pi_1^*\gamma$, with $\gamma\in M$.
Since $V$ and $W_1$ have the same rational span, some positive
integer $c$ satisfies $c\alpha\in V$. Injectivity of $\pi_1^*$ gives
$c\gamma\in M\cap\pi_2^*H^n(X_C,R)_{(1)}=0$.
As $M$ is torsion-free, $\gamma=0$, proving $W_1\subseteq V$.
\end{proof}

Let $U_1\subset H^n(C^n,\mathbb F_\ell)$ be the image of the reduction of
$\operatorname{Alt}$. By the preceding lemma,
$U_1\simeq\bigwedge^nH^1(C,\mathbb F_\ell)$ and the inclusion
$U_1\hookrightarrow H^n(C^n,\mathbb F_\ell)$ is injective.

\begin{proposition}\label{prop:mod-ell-middle-cohomology}
The morphism $\pi_2\circ\pi_1\colon C^n\to X_C$ induces an isomorphism
\[
(\pi_2\circ\pi_1)^*\colon H^n(X_C,\mathbb F_\ell)_{(1)}
 \xrightarrow{\ \sim\ }U_1
 \simeq\bigwedge^nH^1(C,\mathbb F_\ell).
\]
\end{proposition}

\begin{proof}
Since all $H^i(X_C,R)$ are free, the Universal Coefficient Theorem for
change of coefficients, together with the projector $(1-\iota)/2$, gives
$H^n(X_C,\mathbb F_\ell)_{(1)}\simeq
H^n(X_C,R)_{(1)}\otimes_R\mathbb F_\ell$. Since the inclusion
$W_1\hookrightarrow H^n(C^n,R)$ is split, its reduction modulo $\ell$ is
injective with image $U_1$. The result now follows from
\Cref{prop:middle-cohomology-lattice}.
\end{proof}

\subsection{The general case}
Assume that $\mathbb K$ is algebraically closed and that $2\ell$ is
invertible in $\mathbb K$. Let $U_1$ be the image of the same alternating
cup-product map with $\mathbb F_\ell$-coefficients; by K\"unneth,
$U_1\simeq\bigwedge^nH^1(C,\mathbb F_\ell)$.

\begin{proposition}\label{prop:hyperelliptic-comparison}
The quotient morphism $\pi\colon C^n\to X_C$ induces an isomorphism
\[
\pi^*\colon H^n(X_C,\mathbb F_\ell)_{(1)}
 \xrightarrow{\ \sim\ }U_1
 \simeq\bigwedge^nH^1(C,\mathbb F_\ell).
\]
\end{proposition}

\begin{proof}
In characteristic zero, all data descend to a finitely generated subfield
of $\mathbb K$ that embeds in $\mathbb C$. Invariance of \'{e}tale
cohomology under extension of algebraically closed fields
\cite[\href{https://stacks.math.columbia.edu/tag/0F1C}{Lem.~0F1C}]{stacks-project}
and comparison over $\mathbb C$ reduce the assertion to
\Cref{prop:mod-ell-middle-cohomology}.

Suppose now that $\operatorname{char}\mathbb K=p>0$. Put
$R=W(\mathbb K)$ and $S=\operatorname{Spec}R$, and let
$\bar\eta,\bar s$ be geometric
generic and special points of $S$. Choose lifts
$\widetilde p_j=[\widetilde a_j:\widetilde b_j]\in\mathbb P^1(R)$ of the
$p_j$. Since the $p_j$ are distinct,
$\widetilde a_i\widetilde b_j-\widetilde a_j\widetilde b_i$ is a unit for
$i\ne j$; hence the lifted sections are disjoint. The double cover
$h_S\colon\mathcal C\to S$ of $\mathbb P^1_S$ branched along these
sections is therefore smooth and proper, with special fiber $C$.

The corresponding lifted hyperplane has coefficient vector
$v_n(\widetilde a_j,\widetilde b_j)$. For distinct
$j_0,\ldots,j_n$, the homogeneous
Vandermonde identity gives
\[
\det\bigl(\widetilde a_{j_v}^{\,n-u}\widetilde b_{j_v}^{\,u}\bigr)_{0\le u,v\le n}
 =\pm\prod_{0\le u<v\le n}
 (\widetilde a_{j_u}\widetilde b_{j_v}-
  \widetilde a_{j_v}\widetilde b_{j_u})\in R^\times,
\]
so the relative arrangement is in general position. Let
$f_S\colon\mathcal X\to S$ and $F_S\colon\mathcal Y\to S$ be the
associated double and Kummer covers. The relative forms of \Cref{lemma:smooth} and
\Cref{prop:kummer-quotient}, together with the argument in the proof of
\Cref{lem:lisse-cohomology}, show that $F_S$ is smooth and
proper, $\mathcal X\simeq\mathcal Y/\overline N$, and
$R^nf_{S*}\mathbb F_\ell\simeq
(R^nF_{S*}\mathbb F_\ell)^{\overline N}$.
Thus $R^nf_{S*}\mathbb F_\ell$ and
$\mathscr V:=(R^nf_{S*}\mathbb F_\ell)_{(1)}$ are finite locally
constant by smooth proper base change
\cite[\href{https://stacks.math.columbia.edu/tag/0GLI}{Prop.~0GLI}]{stacks-project}.

Let $g_S\colon\mathcal C^n\to S$ be the structure morphism and
$\mathscr L=R^1h_{S*}\mathbb F_\ell$. The product construction gives a
finite $S$-morphism $\Pi\colon\mathcal C^n\to\mathcal X$, whose special
fiber is $\pi$. Relative cup products and K\"unneth define an injection
$\operatorname{Alt}\colon\bigwedge^n\mathscr L\hookrightarrow
R^ng_{S*}\mathbb F_\ell$; denote its finite locally constant image by
$\mathscr U$. Pullback by $\Pi$ gives
$\Phi\colon\mathscr V\to R^ng_{S*}\mathbb F_\ell$.

The quotient $R^ng_{S*}\mathbb F_\ell/\mathscr U$ is finite locally
constant, and the characteristic-zero case shows that the composite of
$\Phi$ with this quotient vanishes at $\bar\eta$; hence it vanishes
identically. Thus $\Phi$ factors through $\mathscr U$. Proper base change
\cite[\href{https://stacks.math.columbia.edu/tag/095T}{Thm.~095T}]{stacks-project},
compatibility of pullback and cup products with specialization, and the
fact that $S$ is strictly henselian give a commutative diagram
\[
\begin{tikzcd}[column sep=large]
H^n(\mathcal X_{\bar\eta},\mathbb F_\ell)_{(1)}
  \arrow[r,"{\Pi_{\bar\eta}^*}"] \arrow[d,"\sim"']
& \mathscr U_{\bar\eta} \arrow[d,"\sim"]\\
H^n(\mathcal X_{\bar s},\mathbb F_\ell)_{(1)}
  \arrow[r,"{\Pi_{\bar s}^*}"']
& \mathscr U_{\bar s}.
\end{tikzcd}
\]
The upper horizontal arrow is an isomorphism by the characteristic-zero
case, hence so is the lower one. Since the special fibers are $X_C$, $C$,
and $\pi$, this is the asserted isomorphism.
\end{proof}

We now globalize the comparison. Let
$\widetilde{\mathfrak M}_1\subset(\mathbb A^2_{\mathbb K})^m$ be the open
subscheme of tuples $(p_1,\dots,p_m)$ with $p_i\ne0$ and pairwise distinct
images in $\mathbb P^1_{\mathbb K}$. Writing $p_i=(a_i,b_i)$, define the
$S_m$-equivariant morphism
\[
i\colon\widetilde{\mathfrak M}_1\longrightarrow\widetilde{\mathfrak M},
\qquad (p_1,\dots,p_m)\longmapsto
(v_n(p_1),\dots,v_n(p_m)).
\]
The Vandermonde identity above shows that $i$ lands in
$\widetilde{\mathfrak M}$. Let
$f_1\colon\widetilde{\mathcal C}\to\widetilde{\mathfrak M}_1$ be the
tautological hyperelliptic family, let
$\mathcal X_1=\widetilde{\mathcal X}\times_{\widetilde{\mathfrak M}}
\widetilde{\mathfrak M}_1$, and let
$g_n\colon\widetilde{\mathcal C}^{\,n}\to\widetilde{\mathfrak M}_1$ be
the $n$-fold fiber product. The product construction gives a finite
morphism $\Pi\colon\widetilde{\mathcal C}^{\,n}\to\mathcal X_1$.

\begin{proposition}\label{prop:hyperelliptic-local-systems}
There is a canonical $S_m$-equivariant isomorphism of finite locally
constant sheaves
\[
i^*(R^n\widetilde f_*\mathbb F_\ell)_{(1)}
 \xrightarrow{\ \sim\ }
\bigwedge^nR^1f_{1*}\mathbb F_\ell.
\]
\end{proposition}

\begin{proof}
Relative cup product and K\"unneth
\cite[\href{https://stacks.math.columbia.edu/tag/0F13}{Sect.~0F13}]{stacks-project}
give a stalkwise injective alternating map
$\bigwedge^nR^1f_{1*}\mathbb F_\ell\to R^ng_{n*}\mathbb F_\ell$; denote
its image by $\mathscr U$. Pullback by $\Pi$ gives a morphism
$\Phi\colon i^*(R^n\widetilde f_*\mathbb F_\ell)_{(1)}\to
R^ng_{n*}\mathbb F_\ell$. By \Cref{prop:hyperelliptic-comparison}, every
geometric stalk of $\Phi$ is an isomorphism onto the corresponding stalk
of $\mathscr U$. Hence $\Phi$ factors through $\mathscr U$ and yields the
stated isomorphism. The construction is natural under simultaneous
permutation of the branch points and hyperplanes, so the isomorphism is
$S_m$-equivariant and descends to the quotient parameter spaces.
\end{proof}

The free $S_m$-actions on $\widetilde{\mathfrak M}_1$ and
$\widetilde{\mathfrak M}$ give finite \'etale Galois quotient maps, and
$i$ descends to
$\bar i\colon\widetilde{\mathfrak M}_1/S_m\to
\widetilde{\mathfrak M}/S_m$. For a geometric point
$s_1\to\widetilde{\mathfrak M}_1$, put $s=i(s_1)$ and denote the quotient
points by $\bar s_1$ and $\bar s$. We have a commutative diagram with exact
rows:
\begin{equation}\label{eq:ordered-unordered-pi1}
\begin{tikzcd}[column sep=small]
1 \arrow[r]
& {\pi_1(\widetilde{\mathfrak M}_1,s_1)}
  \arrow[r] \arrow[d,"{i_*}"]
& {\pi_1(\widetilde{\mathfrak M}_1/S_m,\bar s_1)}
  \arrow[r] \arrow[d,"{\bar i_*}"]
& S_m \arrow[r] \arrow[d,equal]
& 1 \\
1 \arrow[r]
& {\pi_1(\widetilde{\mathfrak M},s)}
  \arrow[r]
& {\pi_1(\widetilde{\mathfrak M}/S_m,\bar s)}
  \arrow[r]
& S_m \arrow[r]
& 1.
\end{tikzcd}
\end{equation}

The double-cover and hyperelliptic families descend to
$\bar f\colon\overline{\mathcal X}\to\widetilde{\mathfrak M}/S_m$ and
$\bar f_1\colon\overline{\mathcal C}\to
\widetilde{\mathfrak M}_1/S_m$, respectively. Let $C$ and $X_C$ be the
fibers over $\bar s_1$ and $\bar s$. The local system
$(R^n\bar f_*\mathbb F_\ell)_{(1)}$ defines
$\rho\colon\pi_1(\widetilde{\mathfrak M}/S_m,\bar s)\to
\operatorname{GL}(H^n(X_C,\mathbb F_\ell)_{(1)})$; we use the same symbol
for all representations induced through \eqref{eq:ordered-unordered-pi1}.
Via $H^n(X_C,\mathbb F_\ell)_{(1)}\simeq
\bigwedge^nH^1(C,\mathbb F_\ell)$, write
$\bigwedge^n\operatorname{Sp}(H^1(C,\mathbb F_\ell))$ for the image of the
exterior-power representation.

\begin{proposition}\label{prop:compare-monodromy-images}
As subgroups of
$\operatorname{GL}(H^n(X_C,\mathbb F_\ell)_{(1)})$, we have
\[
\rho\bigl(\pi_1(\widetilde{\mathfrak M}_1,s_1)\bigr)
=
\rho\bigl(\pi_1(\widetilde{\mathfrak M}_1/S_m,\bar s_1)\bigr)
=
\bigwedge^n\operatorname{Sp}(H^1(C,\mathbb F_\ell)),
\]
and
\[
\rho\bigl(\pi_1(\widetilde{\mathfrak M},s)\bigr)
=
\rho\bigl(\pi_1(\widetilde{\mathfrak M}/S_m,\bar s)\bigr).
\]
\end{proposition}

\begin{proof}
By \Cref{cor:hyperelliptic-universal-monodromy}, the ordered
hyperelliptic family has full symplectic monodromy. The image of the
descended family lies in the same symplectic group and contains the
ordered image, hence is also full. Together with the $S_m$-equivariant
isomorphism of \Cref{prop:hyperelliptic-local-systems}, this proves the
first equality.

For the second, take
$d\in\pi_1(\widetilde{\mathfrak M}/S_m,\bar s)$ and let $\sigma\in S_m$ be
its image. Choose
$a\in\pi_1(\widetilde{\mathfrak M}_1/S_m,\bar s_1)$ with the same image.
By \eqref{eq:ordered-unordered-pi1},
$d\bar i_*(a)^{-1}\in\pi_1(\widetilde{\mathfrak M},s)$; by the first
equality, $\rho(\bar i_*(a))$ also lies in the monodromy image of
$\pi_1(\widetilde{\mathfrak M},s)$. Therefore so does $\rho(d)$. The
reverse inclusion is immediate.
\end{proof}

\section{Main results}\label{sec:main_results}
We first determine the image of
$\rho\colon\pi_1(\widetilde{\mathfrak M},s)
\to\operatorname{GL}(H^n(X,\mathbb F_\ell)_{(1)})$. Recall that
$S=\{\gamma\cdot\delta_{j_0\dots j_n}\mid
\gamma\in\pi_1(\widetilde{\mathfrak M},s),\
1\le j_0<\dots<j_n\le m\}$.

\begin{proposition}\label{prop:single-vanishing-cycle-orbit}
$S$ consists of a single $\pi_1(\widetilde{\mathfrak{M}}, s)$-orbit. In other words, the monodromy group $\rho(\pi_1(\widetilde{\mathfrak{M}}, s))$ acts transitively on $S$.
\end{proposition}

\begin{proof}
Take two boundary components $D_{j_0\dots j_n}$ and
$D_{i_0\dots i_n}$. Choose an $S_m$-stable open subset
$W\subset\mathbb A_\mathbb K^{(n+1)m}$ containing
$\widetilde{\mathfrak M}$ and a dense open subset of each boundary
component, such that the $S_m$-action is free and the boundary is
a disjoint union of smooth divisors. Then $W\to W/S_m$ is finite
\'etale, and the boundary of $\widetilde{\mathfrak M}/S_m$ in
$W/S_m$ is an irreducible smooth divisor $E$.
Such a $W$ is obtained by deleting the singular loci, pairwise
intersections, and the non-free locus; the latter has codimension at
least two and does not remove the generic point of any $D_I$.
The descended local system is tamely ramified along $E$.
Using the same generator $\alpha$ of tame inertia, the two lassos
therefore have conjugate images in
$\pi_1^E(\widetilde{\mathfrak M}/S_m,\bar s)$: the quotient map
has ramification index one along the boundary. Consequently, their
monodromy transformations are conjugate in the unordered monodromy
image. By \Cref{prop:compare-monodromy-images}, this image equals the
ordered one. Hence there exists
$\gamma\in\pi_1(\widetilde{\mathfrak M},s)$ such that
\begin{equation}\label{eq:conjugacy-of-lassos}
 \rho(\gamma_{j_0\dots j_n})
 =\rho(\gamma)\rho(\gamma_{i_0\dots i_n})\rho(\gamma)^{-1}.
\end{equation}

Put $u=\delta_{j_0\dots j_n}$ and
$v=\gamma\cdot\delta_{i_0\dots i_n}$, and set
\[
 c_n\coloneqq
 \begin{cases}
 (-1)^{n/2}2^{m-n-2},&n\text{ even},\\
 (-1)^{(n-1)/2}2^{m-n-2},&n\text{ odd}.
 \end{cases}
\]
Since $\ell$ is odd, $c_n\in\mathbb F_\ell^\times$.  The form $Q$ is
monodromy-invariant, so conjugating the Picard--Lefschetz formula
\eqref{eq:picard-lefschetz} gives
\(\rho(\gamma)T_{\delta_{i_0\dots i_n}}\rho(\gamma)^{-1}=T_v\).
Consequently, after cancelling the nonzero scalar $c_n$, the preceding
equality \eqref{eq:conjugacy-of-lassos} of monodromy
transformations is precisely
\begin{equation}\label{eq:equality-of-rank-one-PL-operators}
 Q(-,u)u=Q(-,v)v.
\end{equation}
We claim that $u=\pm v$.  If $u=0$, then
\eqref{eq:equality-of-rank-one-PL-operators} and the non-degeneracy of
$Q$ imply $v=0$, and the claim is immediate.  Suppose that $u\ne0$.
The image of the rank-one operator $Q(-,u)u$ is $\mathbb F_\ell u$.
Equation \eqref{eq:equality-of-rank-one-PL-operators} therefore gives
$\mathbb F_\ell u=\mathbb F_\ell v$, so $v=\lambda u$ for some
$\lambda\in\mathbb F_\ell^\times$.  Substitution into \eqref{eq:equality-of-rank-one-PL-operators} yields
$Q(-,u)u=\lambda^2Q(-,u)u$. The operator on the left is nonzero because $Q$ is non-degenerate and
$u\ne0$.  Hence $\lambda^2=1$.  Since $\ell\ne2$, we have
$\lambda=\pm1$, proving the claim.

Thus $\delta_{j_0\dots j_n}
=\pm\gamma\cdot\delta_{i_0\dots i_n}$. Finally, $-\operatorname{id}$ belongs to
$\rho(\pi_1(\widetilde{\mathfrak M},s))$ by
\Cref{prop:coefficient-scaling-monodromy}, so the possible minus sign can
be absorbed into the monodromy action.  Hence all the distinguished
vanishing cycles $\delta_{j_0\dots j_n}$ lie in one monodromy orbit,
and therefore their union of orbits $S$ is a single orbit.
\end{proof}

Put $V\coloneqq H^n(X,\mathbb F_\ell)_{(1)}$.  By
\Cref{prop:vanishing-cycles-generate} and
\Cref{prop:single-vanishing-cycle-orbit}, the set $S$ is a single orbit
spanning $V$, while
\Cref{prop:hyperelliptic-local-systems} gives
$\dim V=\binom{m-2}{n}>0$.  Hence every $\delta\in S$ is nonzero.
Throughout the paper, we use
\emph{Picard--Lefschetz transformation} as the parity-independent term
for $T_\delta$.  If $n$ is odd, then $Q$ is alternating and $T_\delta$
is a symplectic transvection.  If $n$ is even, then, by
\Cref{prop:picard-lefschetz} and monodromy invariance,
$q(\delta)=Q(\delta,\delta)/2=(-1)^{n/2}/2^{m-n-2}\ne0$, and
$T_\delta(a)=a-Q(a,\delta)\delta/q(\delta)$ is the orthogonal reflection in the anisotropic vector $\delta$.
In both cases, $T_\delta-\operatorname{id}$ has rank one, so
$T_\delta$ is a pseudoreflection in the terminology of Hall.

Combining the preceding results, we obtain the following irreducibility property.

\begin{proposition}\label{prop:irreducible-monodromy}
The monodromy representation
\[\rho \colon \pi_1(\widetilde{\mathfrak{M}}, s)\to \operatorname{GL}(H^n(X, \mathbb{F}_\ell)_{(1)})\]
 is irreducible.
\end{proposition}

\begin{proof}
 Suppose $W\subset H^n(X, \mathbb{F}_\ell)_{(1)}$ is a nonzero $\pi_1(\widetilde{\mathfrak{M}}, s)$-invariant subspace. Take $0\neq a\in W$. Since $Q$ is non-degenerate and $S$ linearly spans $H^n(X, \mathbb{F}_\ell)_{(1)}$ (\Cref{prop:vanishing-cycles-generate}), there exists $\delta\in S$ with $Q(a,\delta)\neq 0$. Again by \Cref{prop:vanishing-cycles-generate},
the image $\rho(\pi_1(\widetilde{\mathfrak{M}},s))$ contains the
Picard--Lefschetz transformation $T_\delta$.  Hence both $T_\delta(a)$
and $a$ lie in $W$.  Since $T_\delta(a)-a$ is a nonzero scalar
multiple of $\delta$, it follows that $\delta\in W$. Therefore
$S=\pi_1(\widetilde{\mathfrak{M}},s)\cdot \delta\subset W$, and we
conclude that $W=H^n(X, \mathbb{F}_\ell)_{(1)}$.
\end{proof}

Recall that, if $n$ is even, $q(v)=Q(v,v)/2$ and the spinor norm is normalized by $\theta(r_v)=q(v)$ modulo squares. Thus we have the homomorphism
\[
\theta\colon \operatorname{O}\big(H^n(X,\mathbb{F}_\ell)_{(1)}\big)\longrightarrow
\mathbb F_\ell^\times/(\mathbb F_\ell^\times)^2=\{\pm1\},
\]
as well as the product of the spinor norm with the determinant,
\[
\theta\cdot\det\colon
\operatorname{O}\big(H^n(X,\mathbb{F}_\ell)_{(1)}\big)\longrightarrow \{\pm1\}.
\]
\begin{lemma}\label{lem:reflection-groups-char-three}
Let $(V,q)$ be a non-degenerate quadratic space over $\mathbb F_3$, with
$d=\dim V=6$ or $d\ge9$.  Let $G\subseteq\operatorname O(V,q)$ be an
irreducible subgroup generated by orthogonal reflections.  Suppose that
$-\operatorname{id}_V\in G$ and that $G$ contains an element $\sigma\ne1$
with $(\sigma-1)^2=0$.  Then
\[
G\in\{\operatorname O(V,q),\ \ker\theta,\ \ker(\theta\cdot\det)\},
\]
where $\theta$ denotes the spinor norm normalized by
$\theta(r_v)=[q(v)]$ for every anisotropic vector $v$.
\end{lemma}

\begin{proof}
We use the classification of irreducible orthogonal reflection groups over
prime fields as stated in \cite[Thm.~3.7 and Table~1]{etingof2018reflection}.
In the indicated dimensions and in characteristic $3$, its alternatives,
besides the groups of orthogonal type, are the standard representations
$A_{d,3}\simeq S_{d+1}$ when $3\nmid d+1$, the fully deleted permutation
representations $\overline A_{d,3}\simeq S_{d+2}$ of part~(b) when
$3\mid d+2$, and the signed permutation representations $B_{d,3}$ and
$D_{d,3}$.  The dimension restriction excludes every exceptional Coxeter
type except $E_6$, whose six-dimensional representation is not irreducible
in characteristic $3$; see \cite[Sect.~3.3.3]{etingof2018reflection}.
The classification is up to twisted forms, which preserve the abstract
group, scalar $-1$, and the minimal polynomials of its elements; see
\cite[Sect.~3.2]{etingof2018reflection}.  The nontrivial twist rescales the
quadratic form by the nonsquare $-1$.  Reflections are unchanged, while the
product definition of the spinor norm
\cite[p.~76]{grove2002classical} gives
$\theta_{-q}(g)=\theta_q(g)\det(g)$.  Thus this twist interchanges the two
kernels in the conclusion, leaving the list of alternatives unchanged.

The two symmetric-group alternatives have trivial center and so cannot
contain the nontrivial central element $-\operatorname{id}_V$.  In either
signed permutation alternative, the relation $(\sigma-1)^2=0$ gives
$\sigma^3=1$.  A nonidentity signed permutation matrix with this property
must have a $3$-cycle: otherwise it is diagonal with entries $\pm1$, and its
cube can be $1$ only if it is the identity.  The signs along such a cycle
have product $1$, so a diagonal change of basis identifies that block with
the cyclic permutation matrix $P$ on three coordinates.  But over
$\mathbb F_3$,
\[
(P-I)^2=P^2+P+I=
\begin{pmatrix}1&1&1\\1&1&1\\1&1&1\end{pmatrix}\ne0,
\]
contradicting the assumed square-zero relation.

Only the orthogonal-type alternatives remain: the full orthogonal group
and the subgroups generated by its two reflection classes
\cite[Sect.~3.3.1]{etingof2018reflection}.  Put
$H_\varepsilon=\langle r_v\mid q(v)=\varepsilon\rangle$ for
$\varepsilon\in\{1,-1\}$.  We now identify these two subgroups.  In applying
Grove, we use one half of the polar form of $q$ as the symmetric bilinear
form, so that his associated quadratic form is our $q$; see
\cite[pp.~31 and 76]{grove2002classical}.

Since $d\ge6$, the space $V$ is isotropic over $\mathbb F_3$, as follows
from \cite[Prop.~4.8]{grove2002classical}.  By
\cite[Thm.~9.7]{grove2002classical},
$\Omega=\Omega(V)=\operatorname{SO}(V,q)\cap\ker\theta$.
Moreover, $\Omega$ is perfect and
$P\Omega=\Omega/(\Omega\cap\{\pm I_V\})$ is simple by
\cite[Prop.~6.30 and Thm.~6.31]{grove2002classical}; all the exceptional
cases there have dimension at most four.

Each $H_\varepsilon$ is normal in $\operatorname O(V,q)$, since
$g r_v g^{-1}=r_{gv}$ and $q(gv)=q(v)$.  By
\cite[Prop.~4.8]{grove2002classical}, applied first to $V$ and then to the
non-degenerate space $u^\perp$, choose $u\perp v$ with
$q(u)=q(v)=\varepsilon$.  The element $r_ur_v$ belongs to
$N_\varepsilon=H_\varepsilon\cap\Omega$, since it has determinant and spinor
norm both equal to $1$.  It is $-I$ on $\langle u,v\rangle$ and $I$ on its
nonzero orthogonal complement, so $r_ur_v\notin\{\pm I_V\}$.  Thus the
image of $N_\varepsilon$ in $P\Omega$ is a nontrivial normal subgroup and
hence is all of $P\Omega$.  With $Z_0=\Omega\cap\{\pm I_V\}$, it follows that
\[
\Omega=N_\varepsilon Z_0,
\qquad
\Omega=\Omega'=(N_\varepsilon Z_0)'=N_\varepsilon'
\subseteq N_\varepsilon,
\]
where we used the perfectness of $\Omega$ and centrality of $Z_0$.
Consequently, $\Omega\subseteq H_\varepsilon$.  Finally,
$\theta(r_v)=[q(v)]$ and $\det(r_v)=-1$ give
$H_1\subseteq\ker\theta$ and
$H_{-1}\subseteq\ker(\theta\cdot\det)$.
Both kernels have determinant-one subgroup $\Omega$, and each
$H_\varepsilon$ contains a reflection.  These inclusions are therefore
equalities, as required.
\end{proof}

\begin{theorem}\label{thm:coefficient-space-monodromy}
Set
\[
V_\ell\coloneqq H^n(X,\mathbb F_\ell)_{(1)},
\qquad
\widetilde\Gamma
 \coloneqq \rho\!\left(\pi_1(\widetilde{\mathfrak M},s)\right).
\]
Then \(\widetilde\Gamma\) is the following geometric mod-\(\ell\) monodromy
 group.
\begin{compactenum}[\normalfont(1)]
\item If \(n\) is odd, then
\[
\widetilde\Gamma=\operatorname{Sp}(V_\ell,Q).
\]
In this case \(V_\ell=H^n(X,\mathbb F_\ell)\).

\item If \(n\) is even, then
\[
\widetilde\Gamma=
\begin{cases}
\ker\theta,
 & n\equiv0\pmod4,\\[2pt]
\ker\theta,
 & n\equiv2\pmod4\ \text{and}\ \ell\equiv1\pmod4,\\[2pt]
\ker(\theta\cdot\det),
 & n\equiv2\pmod4\ \text{and}\ \ell\equiv3\pmod4.
\end{cases}
\]
Here the kernels are taken inside \(\operatorname O(V_\ell,Q)\).
\end{compactenum}
\end{theorem}
\begin{proof}
Let $V=H^n(X,\mathbb F_\ell)_{(1)}$ and
$\mathscr T=\{T_\delta\mid\delta\in S\}$. By \Cref{prop:vanishing-cycles-generate}, the set
\(\mathscr T\) generates \(\widetilde\Gamma\), and by
\Cref{prop:irreducible-monodromy}, the group \(\widetilde\Gamma\) acts
irreducibly on \(V\).

We first verify the generating-set conditions in
Hall's criterion \cite[Thm.~3.1]{10.1215/S0012-7094-08-14115-8}; its
symplectic part applies for every odd $\ell$, and its orthogonal part for
$\ell\ge5$.
By \Cref{prop:hyperelliptic-local-systems}, one has
\(\dim V=\binom{m-2}{n}>0\).  Every \(\delta\in S\) is therefore
nonzero: indeed, \(S\) is a single monodromy orbit by
\Cref{prop:single-vanishing-cycle-orbit} and spans \(V\) by
\Cref{prop:vanishing-cycles-generate}.
 The Picard--Lefschetz formula gives
 $\operatorname{im}(T_\delta-\operatorname{id})=\mathbb F_\ell\delta$
 and $V^{T_\delta}=\delta^\perp$.
Thus every \(T_\delta\) has drop one.  If \(n\) is odd, it is a
symplectic transvection; if \(n\) is even, it is an orthogonal
reflection.  In either case it is a pseudoreflection in Hall's
terminology. In Hall's notation, take
$r=1$, $S_{\mathrm H}=\mathscr T$, and $S_{\mathrm H,0}=\varnothing$.
Condition~(i) holds because every generator has drop one,
condition~(ii) holds because every generator is a pseudoreflection, and
condition~(iii) is the tautological inequality
$2(r+1)|S_{\mathrm H,0}|=0\le\dim V$. Suppose first that \(n\) is odd.  The pairing on \(V\) is alternating,
and \(\widetilde\Gamma\) contains a symplectic transvection.  Hall's
 theorem now gives $\widetilde\Gamma=\operatorname{Sp}(V)$.

Now suppose that \(n\) is even.  Every element of \(\mathscr T\) is an
orthogonal reflection.  For every odd $\ell$, we construct the isotropic
shear used in Hall's orthogonal criterion and in
\Cref{lem:reflection-groups-char-three}.
Using \Cref{prop:hyperelliptic-local-systems}, write
$V=\bigwedge^nW$, where $W=H^1(C,\mathbb F_\ell)$, equipped with its
symplectic pairing $(\cdot,\cdot)_W$. By
\Cref{prop:compare-monodromy-images}, one has
\(\bigwedge^n\operatorname{Sp}(W)\subset\widetilde\Gamma\).
Choose a nonzero vector \(e\in W\), and let
\(\varphi(a)=a-(a,e)_W e\)
be the corresponding symplectic transvection.  Choose \(h\in W\) with
\((h,e)_W=1\), and put
\(U\coloneqq\langle e,h\rangle^\perp\).
Then $W=\langle e,h\rangle\mathbin{\perp}U$,
$\varphi(e)=e$, $\varphi(h)=h-e$, and
$\varphi|_U=\operatorname{id}$.
Set $\sigma\coloneqq\bigwedge^n\varphi\in\widetilde\Gamma$ and
$D\coloneqq\sigma-\operatorname{id}_V$.
The decomposition of \(W\) induces
\[
\begin{split}
\bigwedge^nW={}&\bigwedge^nU
 \oplus e\wedge\bigwedge^{n-1}U
 \oplus h\wedge\bigwedge^{n-1}U\\
&\oplus e\wedge h\wedge\bigwedge^{n-2}U.
\end{split}
\]
On these four summands, \(D\) is zero except that
$D(h\wedge u)=-e\wedge u$ for $u\in\bigwedge^{n-1}U$.
Consequently, $D^2=0$ and
$\operatorname{rank}D=\binom{m-4}{n-1}>0$.
Indeed, $\dim U=m-4$, and the parity assumptions give $m\ge n+4$.
Thus $\sigma$ is a nontrivial unipotent element satisfying
$(\sigma-1)^2=0$.
Since \(\sigma\in\widetilde\Gamma\subset\operatorname{O}(V,Q)\), it is
an isotropic shear in Hall's sense.

If $\ell\ge5$, Hall's theorem now gives
$\widetilde\Gamma\in
\{\operatorname O(V),\ \ker\theta,\ \ker(\theta\cdot\det)\}$.
If $\ell=3$, put $d=\binom{m-2}{n}$.  Then $d=6$ when $(n,m)=(2,6)$
and $d\ge15$ otherwise: for $n=2$ and $m\ge8$ this follows from
$\binom{m-2}{2}\ge15$, whereas for $n\ge4$ we have
$d\ge\binom{n+2}{2}\ge15$.  Moreover,
$-\operatorname{id}_V\in\widetilde\Gamma$ by
\Cref{prop:coefficient-scaling-monodromy}.  Applying
\Cref{lem:reflection-groups-char-three} to the element $\sigma$ above gives
the same three possibilities.  Thus this list holds for every odd prime
$\ell$.
For every \(\delta\in S\), \Cref{prop:picard-lefschetz}
gives $Q(\delta,\delta)=(-1)^{n/2}/2^{m-n-3}$ and
$q(\delta)=Q(\delta,\delta)/2=(-1)^{n/2}/2^{m-n-2}$.
Since \(m-n-2\) is even, the square class of \(q(\delta)\) is the
square class of \((-1)^{n/2}\).  If this class is trivial, every
generator in \(\mathscr T\) lies in \(\ker\theta\).  On the other hand,
\(\det(T_\delta)=-1\), so these generators do not lie in
\(\ker(\theta\cdot\det)\).  The preceding list therefore forces
\(\widetilde\Gamma=\ker\theta\).  This occurs when
\(n\equiv0\pmod4\), or when \(n\equiv2\pmod4\) and
\(\ell\equiv1\pmod4\).  In the remaining cases,
\(\theta(T_\delta)=-1\) and \(\det(T_\delta)=-1\), so every generator
lies in \(\ker(\theta\cdot\det)\), while no generator lies in
\(\ker\theta\).  The same list then forces
$\widetilde\Gamma=\ker(\theta\cdot\det)$.
\end{proof}

Recall from \Cref{subsection:double_cover} the normalized closed immersion
\[
\jmath\colon\mathfrak M\hookrightarrow\widetilde{\mathfrak M}
\]
and the identity
\[
\mathcal X
 =\widetilde{\mathcal X}
   \times_{\widetilde{\mathfrak M},\jmath}\mathfrak M.
\]
We use the same symbol \(s\) for a geometric base point of \(\mathfrak M\)
and its image in \(\widetilde{\mathfrak M}\).  The monodromy representation
of the tautological family is therefore the composite
\[
\pi_1(\mathfrak M,s)\xrightarrow{\ \jmath_*\ }
\pi_1(\widetilde{\mathfrak M},s)\xrightarrow{\ \rho\ }
\operatorname{GL}\bigl(H^n(X,\mathbb F_\ell)_{(1)}\bigr).
\]
Thus
\[
\Gamma\coloneqq\operatorname{im}(\rho\circ\jmath_*)
 \subseteq\widetilde\Gamma\coloneqq\operatorname{im}\rho.
\]
We will prove that these two groups are equal.  We begin with the following
weaker comparison.
\begin{proposition}\label{prop:compare-normalized-coefficient-images}
 $\widetilde{\Gamma}=\Gamma \cdot \{\pm1\}$.
\end{proposition}
\begin{proof}
Write the tautological matrix on \(\widetilde{\mathfrak M}\) as
\(B=(b_1\mid\cdots\mid b_m)\) and
\(B_0=(b_1\mid\cdots\mid b_{n+1})\).
The determinant of \(B_0\) is a unit.  Put
\(c=B_0^{-1}b_{n+2}=(c_1,\ldots,c_{n+1})^{\mathsf t}\).
Each \(c_i\) is a unit: up to sign, \(c_i\det(B_0)\) is the maximal
minor obtained from \(B_0\) by replacing its \(i\)-th column by
\(b_{n+2}\).  Define
\(P=\operatorname{diag}(c_1^{-1},\ldots,c_{n+1}^{-1})B_0^{-1}\) and
\(\lambda_i=c_i^{-1}\).  Then \(Pb_i=\lambda_i e_i\) for
\(1\le i\le n+1\), while
\(Pb_{n+2}=\mathbf 1=(1,\ldots,1)^{\mathsf t}\).
For \(j\ge n+3\), put \(u_j=Pb_j\), let \(\mu_j\) be the first
coordinate of \(u_j\), and set \(\widehat u_j=\mu_j^{-1}u_j\).  Again
\(\mu_j\) is a unit, because it is, up to multiplication by the already
invertible factors \(\det(B_0)\) and the \(c_i\), a maximal minor of
\(B\).  Thus
\(\widehat u_j=(1,a_{1,j-n-2},\ldots,a_{n,j-n-2})^{\mathsf t}\).
Left multiplication by \(P\) and the indicated column rescalings preserve
the invertibility of all maximal minors, so the entries \(a_{rk}\) define a morphism to \(\mathfrak M\).

These formulas define a morphism
\[
\Phi\colon\widetilde{\mathfrak M}\longrightarrow
\mathfrak M\times\operatorname{GL}_{n+1,\mathbb K}
             \times\mathbb G_{m,\mathbb K}^{m-1},
\]
where the \(m-1\) multiplicative coordinates are
$(\lambda_1,\ldots,\lambda_{n+1},\mu_{n+3},\ldots,\mu_m)$.
Conversely, for every \(\mathbb K\)-scheme \(T\), let
\(a\in\mathfrak M(T)\), let
\(P\in\operatorname{GL}_{n+1,\mathbb K}(T)\), and let the displayed
multiplicative coordinates lie in \(\mathbb G_{m,\mathbb K}(T)\).  If
\(\widehat u_j(a)\) denotes the corresponding normalized column of
\(B(a)\), define
\[
B=P^{-1}
\bigl(
 \lambda_1e_1\mid\cdots\mid\lambda_{n+1}e_{n+1}
 \mid\mathbf1
 \mid\mu_{n+3}\widehat u_{n+3}(a)
 \mid\cdots\mid\mu_m\widehat u_m(a)
\bigr).
\]
This gives the inverse of \(\Phi\).  All denominators occurring above are
maximal minors inverted in the coordinate ring of
\(\widetilde{\mathfrak M}\); hence both constructions are morphisms of
\(\mathbb K\)-schemes.  We have therefore obtained an isomorphism
\begin{equation}\label{eq:coefficient-space-product-decomposition}
\Phi\colon\widetilde{\mathfrak M}\xrightarrow{\ \sim\ }
\mathfrak M\times\operatorname{GL}_{n+1,\mathbb K}
             \times\mathbb G_{m,\mathbb K}^{m-1}.
\end{equation}

Under this isomorphism, let
\[
\widetilde\pi
 =\Phi^{-1}\circ
 \bigl(\operatorname{id}_{\mathfrak M}\times
       \operatorname{id}_{\operatorname{GL}_{n+1,\mathbb K}}\times
       [2]^{m-1}\bigr)\circ\Phi,
\]
where \([2]\colon\mathbb G_m\to\mathbb G_m\) is squaring.  This is a
finite \'etale Galois self-cover of \(\widetilde{\mathfrak M}\) with deck
group scheme \(\mu_2^{m-1}\); on geometric points its deck group is
\(\{\pm1\}^{m-1}\).  After changing the geometric base point, which only
conjugates the monodromy image, we may assume that
\(\Phi(s)=(a,I_{n+1},1,\ldots,1)\).
The covering then gives an exact sequence
\[
1\longrightarrow\pi_1(\widetilde{\mathfrak M},s)
 \xrightarrow{\ \widetilde\pi_*\ }
\pi_1(\widetilde{\mathfrak M},s)
 \longrightarrow\{\pm1\}^{m-1}\longrightarrow1.
\]

Let
\[
p_1\colon
\mathfrak M\times\operatorname{GL}_{n+1,\mathbb K}
             \times\mathbb G_{m,\mathbb K}^{m-1}
\longrightarrow\mathfrak M
\]
be the projection.  The projective coordinate change encoded by the \(\operatorname{GL}_{n+1}\)-factor,
together with rescaling the double-cover coordinate by the product of the
multiplicative coordinates, identifies \(\widetilde\pi^*\widetilde{\mathcal X}\)
with \(\Phi^*p_1^*\mathcal X\).  Consequently, the image of
\(\operatorname{im}\widetilde\pi_*\) under \(\rho\) is contained in
\(\Gamma\).  We now describe generators of the quotient
\(\pi_1(\widetilde{\mathfrak M},s)/\operatorname{im}\widetilde\pi_*\).
For each \(1\le i\le m-1\), define
\[
\begin{aligned}
\varphi_i\colon\mathbb G_{m,\mathbb K}
&\longrightarrow\widetilde{\mathfrak M},\\
t&\longmapsto
\Phi^{-1}(a,I_{n+1},1,\ldots,t,\ldots,1),
\end{aligned}
\]
where \(t\) occupies the \(i\)-th multiplicative coordinate.  For each \(i\), let
\(\eta_i\in\pi_1(\mathbb G_{m,\mathbb K},1)\) be the element \(\eta\)
of \Cref{prop:coefficient-scaling-monodromy}.  The images of
\(\varphi_{i*}\eta_i\), \(1\le i\le m-1\), form the standard generating
set of the quotient \(\{\pm1\}^{m-1}\).  Hence every coset of
\(\operatorname{im}\widetilde\pi_*\) is represented by a product
\(\prod_{i=1}^{m-1}(\varphi_{i*}\eta_i)^{\epsilon_i}\) with
\(\epsilon_i\in\{0,1\}\).
It follows that
\(\widetilde\Gamma=\rho(\pi_1(\widetilde{\mathfrak M},s))\) is generated
by \(\Gamma\) together with the elements \(\rho(\varphi_{i*}\eta_i)\).
By \Cref{prop:coefficient-scaling-monodromy}, each of these elements is
\(-\operatorname{id}\), and therefore
\(\widetilde{\Gamma}=\Gamma\cdot\{\pm\operatorname{id}\}\).
\end{proof}

\begin{lemma}\label{lem:Gamma-contains-reflection}
Suppose that $n$ is even. Then the group $\Gamma$ contains a
Picard--Lefschetz reflection.
\end{lemma}
\begin{proof}
Let $\mathfrak M^{\mathrm{aff}}=\mathbb A_{\mathbb K}^{n(m-n-2)}$
be the affine space containing $\mathfrak M$ through the normalized matrix
\eqref{eq:normalized-arrangement-matrix}, and let
\[
\jmath_{\mathrm{aff}}\colon \mathfrak M^{\mathrm{aff}}
\longrightarrow \mathbb A_{\mathbb K}^{(n+1)m}
\]
be the affine-linear morphism obtained by inserting the normalized columns.
Denote by
\(E\coloneqq\{a_{n1}=0\}\subset\mathfrak M^{\mathrm{aff}}\)
the corresponding coordinate hyperplane, and put
\(I\coloneqq\{1,\ldots,n,n+3\}\) and
\(D_I\coloneqq D_{1,\ldots,n,n+3}\).
The maximal minor of the normalized matrix indexed by $I$ is $a_{n1}$.

There is a nonempty open subset $E^{\circ}\subset E$ on which every
other maximal minor is nonzero. Indeed, one may first choose
\(1,a_{11},\ldots,a_{n-1,1},0\)
to be pairwise distinct. Among the first $n+3$ columns, the maximal
minors containing the first variable column and $n$ standard columns are
its coordinates, while those containing the all-ones column, the first
variable column, and $n-1$ standard columns are differences of two of
these coordinates. Thus the minor indexed by $I$ is the only one that
vanishes. The remaining variable columns may then be chosen successively
outside a finite union of proper hyperplanes.

Fix a \(\mathbb K\)-point \(t\in E^{\circ}(\mathbb K)\).  The
coordinate line through \(t\)
obtained by varying only $a_{n1}$ meets $E$ transversally; after shrinking
it around $t$, it meets the boundary
$\mathfrak M^{\mathrm{aff}}\setminus\mathfrak M$ only at $t$. Moreover,
$\jmath_{\mathrm{aff}}(t)$ lies in the regular locus of $D_I$ and on no other irreducible
component of $D_0$. Since
\(\jmath_{\mathrm{aff}}^*\det(t_{r,k})_{0\le r\le n,\ k\in I}=\pm a_{n1}\),
the induced curve in $\mathbb A_{\mathbb K}^{(n+1)m}$ meets $D_I$
transversally. Equivalently, on the corresponding strict henselian
traits, the pullback of a local parameter of $D_I$ is a uniformizer, so
the induced map on tame inertia has ramification index one. Choose a lift
of the fixed tame generator $\alpha$ to the fundamental group of this
punctured strict henselian trait, and let
$\gamma_E\in\pi_1(\mathfrak M,s)$ be its image after choosing a path to
$s$. The image of $\gamma_E$ under
\[
\pi_1(\mathfrak M,s)
 \xrightarrow{\jmath_*}\pi_1(\widetilde{\mathfrak M},s)
 \longrightarrow
 \pi_1^{D_0}(\widetilde{\mathfrak M},s)
\]
is conjugate to the lasso $\gamma_I$ around $D_I$, where
$\gamma_I=\gamma_{1,\ldots,n,n+3}$ denotes the fixed lasso chosen above.

By \Cref{prop:picard-lefschetz}, there is a vanishing
cycle $\delta\in S\subset H^n(X,\mathbb F_\ell)_{(1)}$ such that
\(\rho(\jmath_*\gamma_E)(a)=a-(-1)^{n/2}2^{m-n-2}Q(a,\delta)\,\delta\).
Together with
\(q(\delta)=Q(\delta,\delta)/2=(-1)^{n/2}/2^{m-n-2}\),
this gives
\(\rho(\jmath_*\gamma_E)(a)=a-Q(a,\delta)\delta/q(\delta)\),
which is the orthogonal reflection in the anisotropic vector $\delta$. Since
$\rho(\jmath_*\gamma_E)\in\Gamma$, the assertion follows.
\end{proof}

Now we are ready to prove our main theorem.
\begin{proof}[Proof of \Cref{thm:main-monodromy}]
Let \(V = H^{n}(X, \mathbb{F}_{\ell})_{(1)}\).
By \Cref{prop:compare-normalized-coefficient-images}, the index
\([\widetilde{\Gamma} : \Gamma] \le 2\).
Hence \(\Gamma\) is a normal subgroup of \(\widetilde{\Gamma}\), and the quotient
\(\widetilde{\Gamma}/\Gamma\) is either trivial or isomorphic to
\(\mathbb{Z}/2\mathbb{Z}\).
In particular, \(\Gamma\) contains the derived subgroup
\(\widetilde{\Gamma}'\), and moreover
\(\{\tau^{2} \mid \tau \in \widetilde{\Gamma}\}\subset \Gamma\).

If $n$ is odd, then $\widetilde{\Gamma}=\operatorname{Sp}(V)$ by \Cref{thm:coefficient-space-monodromy}.
For any $\delta\in V$ and $\lambda\in\mathbb{F}_\ell$, let
$T_{\delta,\lambda}\in\operatorname{Sp}(V)$ be the symplectic
transvection defined by $T_{\delta,\lambda}(a)=a+\lambda Q(a,\delta)\,\delta$.
Since $T_{\delta,\lambda}=T_{\delta,\lambda/2}^{\,2}$, we have $T_{\delta,\lambda}\in\Gamma$.
As $\operatorname{Sp}(V)$ is generated by all such symplectic transvections
$\{T_{\delta,\lambda}\mid \delta\in V,\ \lambda\in\mathbb{F}_\ell\}$,
we conclude that $\Gamma=\operatorname{Sp}(V)$ in this case.

Suppose that $n$ is even.  By \Cref{lem:Gamma-contains-reflection},
there is a Picard--Lefschetz reflection $T_{\delta_0}\in\Gamma$ with
$\delta_0\in S$.  Since $S$ is a single $\widetilde\Gamma$-orbit by
\Cref{prop:single-vanishing-cycle-orbit}, every $T_\delta$, $\delta\in S$,
is conjugate to $T_{\delta_0}$ in $\widetilde\Gamma$.  Normality of
$\Gamma$ therefore implies that every $T_\delta$ belongs to $\Gamma$.
These transformations generate $\widetilde\Gamma$ by
\Cref{prop:vanishing-cycles-generate}, so
$\Gamma=\widetilde\Gamma$.  The assertion follows from
\Cref{thm:coefficient-space-monodromy} for every odd prime $\ell$.
\end{proof}

\subsection{Integral \texorpdfstring{$\ell$}{l}-adic monodromy}
\label{subsec:integral-ell-adic-monodromy}

Let
\[
F_{\mathfrak M}\colon
\mathcal Y_{\mathfrak M}
 =\mathcal Y\times_{\widetilde{\mathfrak M}}\mathfrak M
 \longrightarrow\mathfrak M
\]
be the restriction of the smooth Kummer family.  The full sign-change group
acts on $R^nF_{\mathfrak M*}\mathbb Z_\ell$, and we set
\[
\mathcal L_\ell
 \coloneqq
 \left(\left(R^nF_{\mathfrak M*}\mathbb Z_\ell\right)^{\overline N}
 \right)_{(1)}.
\]
The finite-quotient argument of \Cref{lem:lisse-cohomology}, applied
with $\mathbb Z/\ell^\nu\mathbb Z$-coefficients for every $\nu$, identifies
$\mathcal L_\ell/\ell^\nu\mathcal L_\ell$ with
$\left(R^nf_*(\mathbb Z/\ell^\nu\mathbb Z)\right)_{(1)}$.
For the fixed geometric base point $s$, put
\[
L_\ell\coloneqq(\mathcal L_\ell)_s,
\qquad
\overline L_\ell\coloneqq L_\ell/\ell L_\ell,
\]
and let $Q_\ell$ be the pairing on $L_\ell$ induced by Poincar\'e duality on
the smooth Kummer cover.  As before, we suppress its Tate twist.  When $n$
is even, choose a generator of $\mathbb Z_\ell(-n/2)$ lifting the choice made
modulo $\ell$ in the introduction, and use its square to trivialize
$\mathbb Z_\ell(-n)$.  Replacing this generator rescales $Q_\ell$ by a
square unit, so none of the spinor-norm statements below is affected.

\begin{proposition}\label{prop:integral-lattice-and-reduction}
The module $L_\ell$ is finite free over $\mathbb Z_\ell$, and $Q_\ell$ is a
perfect pairing.  Reduction modulo $\ell$ induces an isomorphism
\[
(L_\ell/\ell L_\ell,\overline Q_\ell)
 \xrightarrow{\ \sim\ }
\bigl(H^n(X,\mathbb F_\ell)_{(1)},Q\bigr).
\]
Consequently, if
\[
K_\ell\coloneqq\operatorname{im}\!\left(
\rho^\ell\colon\pi_1(\mathfrak M,s)
 \longrightarrow\operatorname{Aut}(L_\ell,Q_\ell)
\right),
\]
then $K_\ell$ is a closed subgroup and its reduction modulo $\ell$ is the
finite monodromy group $\Gamma$ of \Cref{thm:main-monodromy}.
\end{proposition}

\begin{proof}
It is enough to verify the assertions on a geometric fiber.  If
$\operatorname{char}\mathbb K=0$, descend $Y$, together with the sign-change
action, to a finitely generated subfield of $\mathbb K$ admitting an embedding
into $\mathbb C$.  If $\operatorname{char}\mathbb K=p>0$, lift the coefficients
of the defining hyperplanes to the Witt ring $W(\mathbb K)$.  Since every
maximal minor is nonzero on the special fiber, the lifted minors are units,
and the construction of \Cref{lemma:smooth} gives a smooth proper complete
intersection lifting of $Y$.  Smooth proper base change identifies the
$\mathbb Z_\ell$-cohomology of its special and geometric generic fibers; the
latter can again be descended to a characteristic-zero field admitting an
embedding into $\mathbb C$.  The comparison theorem therefore identifies
$H^n(Y,\mathbb Z_\ell)$ with the $\ell$-adic completion of the integral
cohomology of a smooth complex complete intersection.  The integral Lefschetz
hyperplane theorem, Poincar\'e duality, and the universal coefficient theorem
show that this integral cohomology is torsion-free.  Hence
$H^n(Y,\mathbb Z_\ell)$ is finite free and its middle Poincar\'e pairing is
perfect.

Both projectors
\[
e_{\overline N}=\frac1{|\overline N|}
 \sum_{g\in\overline N}g,
\qquad
 e_- =\frac{1-\iota}{2}
\]
are defined over $\mathbb Z_\ell$, since $\ell$ is odd, and are self-adjoint
for the Poincar\'e pairing.  Their common image is therefore an orthogonal
direct summand on which the pairing remains perfect; this image is precisely
$L_\ell$.  The same projectors commute with reduction modulo $\ell$, giving
the displayed identification with the mod-$\ell$ negative eigenspace.  The
last assertion follows because the representation on this reduction is the
mod-$\ell$ representation already considered in
\Cref{thm:main-monodromy}.  Finally, the continuous image of the profinite
group $\pi_1(\mathfrak M,s)$ is compact and hence closed.
\end{proof}

For the orthogonal case we shall use the following integral form of the spin
cover.  Write $q_\ell(v)=Q_\ell(v,v)/2$.

\begin{lemma}\label{lem:integral-spin-cover}
Suppose that $n$ is even.  Let
\[
\operatorname{red}_\ell\colon
\operatorname{SO}(L_\ell,q_\ell)(\mathbb Z_\ell)
 \longrightarrow
\operatorname{SO}(\overline L_\ell,\overline q_\ell)(\mathbb F_\ell)
\]
be reduction, and put
\[
\Omega(\overline L_\ell)
 \coloneqq
\ker\theta\cap
\operatorname{SO}(\overline L_\ell,\overline q_\ell)(\mathbb F_\ell).
\]
Then
\[
\operatorname{im}\!\left(
\operatorname{Spin}(L_\ell,q_\ell)(\mathbb Z_\ell)
 \longrightarrow
\operatorname{SO}(L_\ell,q_\ell)(\mathbb Z_\ell)
\right)
=
\operatorname{red}_\ell^{-1}\bigl(\Omega(\overline L_\ell)\bigr).
\]
\end{lemma}

\begin{proof}
The central extension
\[
1\longrightarrow\boldsymbol\mu_2
 \longrightarrow\operatorname{Spin}(L_\ell,q_\ell)
 \longrightarrow\operatorname{SO}(L_\ell,q_\ell)
 \longrightarrow1
\]
and its connecting homomorphism give the integral spinor norm; see
\cite[Chap.~IV, Sects.~5--6]{knus1991quadratic}.  The same construction over
$\mathbb F_\ell$ gives the usual finite-field spinor norm, whose kernel in the
special orthogonal group is $\Omega(\overline L_\ell)$.  Since $\ell$ is odd,
Hensel's lemma gives an isomorphism
\[
\mathbb Z_\ell^\times/(\mathbb Z_\ell^\times)^2
 \xrightarrow{\ \sim\ }
\mathbb F_\ell^\times/(\mathbb F_\ell^\times)^2.
\]
Naturality of the two connecting homomorphisms therefore shows that an
integral special orthogonal transformation lifts to the spin group if and
only if its reduction has trivial spinor norm.  This is the claimed equality.
\end{proof}

\begin{theorem}\label{thm:ell-adic-monodromy}
Let $K_\ell=\operatorname{im}\rho^\ell$ be as in
\Cref{prop:integral-lattice-and-reduction}.
\begin{compactenum}[\normalfont(1)]
\item If $n$ is odd, then for every odd prime $\ell$,
\[
K_\ell=\operatorname{Sp}(L_\ell,Q_\ell)(\mathbb Z_\ell).
\]
\item Suppose that $n$ is even.  For every odd prime $\ell$, let
\[
\operatorname{red}_\ell\colon
\operatorname O(L_\ell,q_\ell)(\mathbb Z_\ell)
 \longrightarrow
\operatorname O(\overline L_\ell,\overline q_\ell)(\mathbb F_\ell)
\]
be reduction modulo $\ell$.  Then
\[
K_\ell
 =\operatorname{red}_\ell^{-1}(\Gamma),
\]
where $\Gamma$ is the mod-$\ell$ group of
\Cref{thm:main-monodromy}.  Equivalently, if
\[
\Theta_\ell\coloneqq
\theta\circ\operatorname{red}_\ell\colon
\operatorname O(L_\ell,q_\ell)(\mathbb Z_\ell)
 \longrightarrow\{\pm1\},
\]
then
\[
K_\ell=
\begin{cases}
\ker\Theta_\ell,
 & n\equiv0\pmod4,\\[2pt]
\ker\Theta_\ell,
 & n\equiv2\pmod4\text{ and }\ell\equiv1\pmod4,\\[2pt]
\ker(\Theta_\ell\cdot\det),
 & n\equiv2\pmod4\text{ and }\ell\equiv3\pmod4.
\end{cases}
\]
\end{compactenum}
\end{theorem}

\begin{proof}
By \Cref{prop:integral-lattice-and-reduction}, the reduction of $K_\ell$ is
exactly the group $\Gamma$ determined in \Cref{thm:main-monodromy}.

\smallskip
\noindent\emph{The case where $n$ is odd.}
By \Cref{prop:hyperelliptic-comparison,prop:integral-lattice-and-reduction},
put
\[
 d:=\operatorname{rank}_{\mathbb Z_\ell}L_\ell=\binom{m-2}{n},
 \qquad G_\ell:=\operatorname{Sp}(L_\ell,Q_\ell).
\]
The group scheme $G_\ell$ is semisimple and simply connected, so
$c(G_\ell)=1$ in Vasiu's notation. The closed subgroup $K_\ell$
surjects onto $G_\ell(\mathbb F_\ell)$ by the mod-$\ell$ result.
If $\ell\ge5$, then
\cite[Thm.~1.3(a)]{vasiu2003surjectivity} gives
$K_\ell=G_\ell(\mathbb Z_\ell)$.
If $\ell=3$ and $d\ge4$, the special fiber has type $C_{d/2}$
with $d/2\ge2$ and has no normal subgroup which is an
$\operatorname{SL}_2$- or $\operatorname{PGL}_2$-group.
Thus the additional condition in
\cite[Thm.~1.3(b)]{vasiu2003surjectivity} is vacuous,
and the same conclusion follows.

The only remaining case is $(n,m,\ell)=(1,4,3)$:
indeed, $d=m-2$ when $n=1$, whereas $d\ge n+1\ge4$ when $n\ge3$.
Writing $a=a_{11}$, the normalized family is
\[
 w^2=x_0x_1(x_0+x_1)(x_0+a x_1).
\]
On the chart $x_1=1$, the substitutions $z=-x_0$ and $y=\epsilon w$,
where $\epsilon\in\mathbb K$ satisfies $\epsilon^2=-1$, identify this
family with the Legendre family
\[
 E_a:\quad y^2=z(z-1)(z-a),
 \qquad a\in\mathbb A^1_{\mathbb K}\setminus\{0,1\}.
\]
Its invariants are
\[
 \Delta=16a^2(1-a)^2,\qquad c_4=16(a^2-a+1).
\]
Hence over $F:=\mathbb K((a))$ it has split multiplicative reduction
of type $I_2$. By Tate uniformization
\cite[Chap.~V]{silverman1994advanced}, its Tate parameter $q_E$
has valuation $2$. Since $\mathbb K$ is algebraically closed and
$\operatorname{char}\mathbb K\ne3$, all $3$-power roots of unity
lie in $F$, and the unit part of $q_E$ admits compatible $3$-power
roots in $F$. Consequently, in a suitable basis, the local Galois
image on $T_3(E_a)$ contains
\[
 T=\begin{pmatrix}1&2\\0&1\end{pmatrix}.
\]
(See also \cite[Cor.~3.7(iii)]{bandini2009torsion}
for this local unipotent calculation in positive
residue characteristic.)
Passing to the dual gives a conjugate matrix on $L_3$, so $K_3$
contains such an element as well. Its cube satisfies
\[
 T^3=\begin{pmatrix}1&6\\0&1\end{pmatrix}
 \equiv I\pmod3,\qquad T^3\not\equiv I\pmod9.
\]
Thus the image of $K_3$ modulo $9$ meets
\[
 \ker\bigl(\operatorname{SL}_2(\mathbb Z/9\mathbb Z)
        \longrightarrow\operatorname{SL}_2(\mathbb F_3)\bigr)
\]
nontrivially. Thus the extra hypothesis of
\cite[Thm.~1.3(b)]{vasiu2003surjectivity}
holds for $G_3=\operatorname{SL}_2$, giving
$K_3=\operatorname{SL}_2(\mathbb Z_3)$.

\smallskip
\noindent\emph{The case where $n$ is even.}
Let $\ell$ be any odd prime and write
\(K_\ell^+\coloneqq K_\ell\cap
\operatorname{SO}(L_\ell,q_\ell)(\mathbb Z_\ell)\).
Since both possible groups in \Cref{thm:main-monodromy}\hyperref[item:main-monodromy-even]{(\ref*{item:main-monodromy-even})} have
intersection
\[
\Gamma\cap\operatorname{SO}(\overline L_\ell,\overline q_\ell)
 =\Omega(\overline L_\ell),
\]
we have
\[
\operatorname{red}_\ell(K_\ell^+)=\Omega(\overline L_\ell).
\]
Indeed, an element of $K_\ell$ reducing to determinant one already has
integral determinant one, because $\ell$ is odd.

Let
\[
\pi_\ell\colon
\operatorname{Spin}(L_\ell,q_\ell)
 \longrightarrow
\operatorname{SO}(L_\ell,q_\ell)
\]
be the spin covering, and define the closed subgroup
\[
\widetilde K_\ell
 \coloneqq
\pi_\ell^{-1}(K_\ell^+)
 \subset
\operatorname{Spin}(L_\ell,q_\ell)(\mathbb Z_\ell).
\]
We claim that $\widetilde K_\ell$ surjects onto
$\operatorname{Spin}(\overline L_\ell,\overline q_\ell)(\mathbb F_\ell)$.
Let $\bar s$ be a point of the latter group and set
$\bar g=\pi_\ell(\bar s)\in\Omega(\overline L_\ell)$.  Choose
$g\in K_\ell^+$ reducing to $\bar g$.  By
\Cref{lem:integral-spin-cover}, $g$ lifts to an element $s$ of the integral
spin group.  The reductions of $s$ and $\bar s$ differ by an element of
$\boldsymbol\mu_2(\mathbb F_\ell)$, which lifts to
$\boldsymbol\mu_2(\mathbb Z_\ell)$; after changing $s$ by this element, its
reduction is $\bar s$.  This proves the claim.

The group scheme $\operatorname{Spin}(L_\ell,q_\ell)$ is semisimple and
simply connected, so its Vasiu invariant $c(G)$ equals $1$.  For $\ell\ge5$,
we apply \cite[Thm.~1.3(a)]{vasiu2003surjectivity}.  For $\ell=3$, its
rank-$d$ quadratic module has $d=\binom{m-2}{n}\ge6$.  The geometric
special fiber has type $D_{d/2}$ when $d$ is even (with $D_3=A_3$ at $d=6$)
and type $B_{(d-1)/2}$ when $d$ is odd.  It therefore has no normal
$\operatorname{SL}_2$- or $\operatorname{PGL}_2$-group, so the additional
condition in \cite[Thm.~1.3(b)]{vasiu2003surjectivity} is vacuous.
In either case,
\(\widetilde K_\ell=\operatorname{Spin}(L_\ell,q_\ell)(\mathbb Z_\ell)\).
Taking images under $\pi_\ell$ and using
\Cref{lem:integral-spin-cover}, we obtain
\[
K_\ell^+
 =\operatorname{red}_\ell^{-1}
   \bigl(\Omega(\overline L_\ell)\bigr)
 \subset\operatorname{SO}(L_\ell,q_\ell)(\mathbb Z_\ell).
\]

Set
\(H_\ell\coloneqq\operatorname{red}_\ell^{-1}(\Gamma)
 \subset\operatorname O(L_\ell,q_\ell)(\mathbb Z_\ell)\).
The preceding equality gives
$H_\ell\cap\operatorname{SO}(L_\ell,q_\ell)=K_\ell^+$.  Since $\Gamma$
contains Picard--Lefschetz reflections, choose $r\in K_\ell$ whose reduction
has determinant $-1$; then $\det r=-1$.  Both $K_\ell$ and $H_\ell$ are the
union of their common determinant-one subgroup $K_\ell^+$ and its coset
$rK_\ell^+$.  Hence $K_\ell=H_\ell$, proving the inverse-image formula.
The explicit description by $\Theta_\ell$ follows immediately from
\Cref{thm:main-monodromy}\hyperref[item:main-monodromy-even]{(\ref*{item:main-monodromy-even})}.
\end{proof}

\section{Some consequences}\label{sec:some_consequences}

Throughout this section, all parameter spaces and families are defined over a
finite field $\mathbb F_q$ of characteristic $p$.  Write
$f\colon\mathcal X\to\mathfrak M$ and
$F_{\mathfrak M}\colon\mathcal Y_{\mathfrak M}\to\mathfrak M$ for the
double-cover and Kummer families, and, for every odd prime $\ell\ne p$, set
$\mathcal L_\ell=((R^nF_{\mathfrak M*}\mathbb Z_\ell)^{\overline N})_{(1)}$ as
in \Cref{subsec:integral-ell-adic-monodromy}.  If
$u\in\mathfrak M(\mathbb F_{q^i})$ and $\bar u$ is a geometric point above
$u$, then
$(\mathcal L_\ell)_{\bar u}=H^n(\mathcal X_{\bar u},\mathbb Z_\ell)_{(1)}$,
where
$\mathcal X_{\bar u}=\mathcal X_u\times_{\mathbb F_{q^i}}
\overline{\mathbb F}_q$.  We denote geometric Frobenius at $u$ by
$\operatorname{Fr}_{u,i}$.

\begin{proposition}\label{prop:finite-field-compatible-system}
The scheme $\mathfrak M/\mathbb F_q$ is smooth, geometrically integral, and
of positive dimension.  Each $\mathcal L_\ell$ is torsion-free and lisse,
$\mathcal L_\ell\otimes\mathbb Q_\ell$ is punctually pure of weight $n$, and
the Poincar\'e pairing on the Kummer family induces a perfect pairing
$\mathcal L_\ell\otimes\mathcal L_\ell\to\mathbb Z_\ell(-n)$, alternating
for odd $n$ and symmetric for even $n$.

The sheaves $\{\mathcal L_\ell\}_{\ell\ne2,p}$ form a compatible system in
the sense used by Chavdarov.  Namely, for every
$u\in\mathfrak M(\mathbb F_{q^i})$, the polynomial
\[
 P_u(T)\coloneqq\det\!\left(1-T\operatorname{Fr}_{u,i}\mid
 (\mathcal L_\ell)_{\bar u}\otimes\mathbb Q_\ell\right)
\]
 belongs to $\mathbb Z[T]$ and is independent of $\ell$.  Writing
 $Z(\mathcal X_u/\mathbb F_{q^i},T)$ for the Hasse--Weil zeta function
 and setting $D_{q^i}(T)=\prod_{j=0}^{n}(1-q^{ij}T)$, we have
\[
 Z(\mathcal X_u/\mathbb F_{q^i},T)=
 \begin{cases}
 P_u(T)/D_{q^i}(T),&n\text{ odd},\\[4pt]
 1/(P_u(T)D_{q^i}(T)),&n\text{ even}.
 \end{cases}
\]
\end{proposition}

\begin{proof}
By \Cref{subsection:double_cover}, $\mathfrak M$ is a nonempty open subscheme
of $\mathbb A_{\mathbb F_q}^{\,n(m-n-2)}$; geometric nonemptiness also
follows from the rational-normal-curve locus of
\Cref{sec:connections_with_hyperelliptic_curves}.  Since
$n(m-n-2)>0$, the assertion about $\mathfrak M$ follows.

The lissity, torsion-freeness, and perfect pairing are the family versions of
\Cref{prop:integral-lattice-and-reduction}.  Moreover,
$\mathcal L_\ell\otimes\mathbb Q_\ell$ is the image of commuting idempotents
on $R^nF_{\mathfrak M*}\mathbb Q_\ell$; hence smooth properness of
$F_{\mathfrak M}$ and Weil II imply purity of weight $n$
\cite{deligne1980weil2}.

The argument of \Cref{prop:cohomology-away-middle} with $\mathbb Q_\ell$
coefficients gives
$H^j(\mathcal X_{\bar u},\mathbb Q_\ell)\simeq
H^j(\mathbb P^n_{\overline{\mathbb F}_q},\mathbb Q_\ell)$ for $j\ne n$,
and identifies the $+1$-eigenspace in degree $n$ with
$H^n(\mathbb P^n_{\overline{\mathbb F}_q},\mathbb Q_\ell)$.  The trace
formula therefore gives the stated zeta-function identity.  Since the zeta
function and $D_{q^i}(T)$ are independent of $\ell$, so is $P_u(T)$; its
coefficients are rational algebraic integers, being symmetric polynomials in
Frobenius eigenvalues on the smooth proper Kummer cover.  Thus
$P_u(T)\in\mathbb Z[T]$, proving compatibility.
\end{proof}

\begin{proposition}\label{prop:generic-zeta-irreducibility}
Suppose that $n$ is odd.  Then
\[
\lim_{i\to\infty}
\frac{\#\{u\in\mathfrak M(\mathbb F_{q^i})\mid
P_u(T)\text{ is irreducible over }\mathbb Q\}}
{\#\mathfrak M(\mathbb F_{q^i})}=1.
\]
Equivalently, the numerator of $Z(\mathcal X_u/\mathbb F_{q^i},T)$ is
irreducible over $\mathbb Q$ for a proportion tending to $1$ of the points
$u\in\mathfrak M(\mathbb F_{q^i})$.
\end{proposition}

\begin{proof}
Put $2b=\operatorname{rank}_{\mathbb Z_\ell}\mathcal L_\ell=\binom{m-2}{n}$.  By
\Cref{prop:finite-field-compatible-system}, the sheaves $\mathcal L_\ell$
form, away from $2p$, a compatible system of torsion-free lisse sheaves of
rank $2b$ with a nondegenerate alternating, geometrically invariant duality
$\mathcal L_\ell^{\otimes2}\to\mathbb Z_\ell(-n)$.  By
\Cref{thm:main-monodromy}\hyperref[item:main-monodromy-odd]{(\ref*{item:main-monodromy-odd})}, their geometric mod-$\ell$ monodromy is
$\operatorname{Sp}_{2b}(\mathbb F_\ell)$ for every odd $\ell\ne p$.
Chavdarov's compatible-system theorem
\cite[Thm.~2.1]{chavdarov1997generic} therefore gives the asserted
irreducibility; the zeta-function formulation follows from
\Cref{prop:finite-field-compatible-system}.
\end{proof}

Fix an odd prime $\ell\ne p$ and put
$\mathcal V_\ell=\mathcal L_\ell/\ell\mathcal L_\ell$.  Its stalk at $\bar u$
is $H^n(\mathcal X_{\bar u},\mathbb F_\ell)_{(1)}$.

\begin{proposition}\label{prop:frobenius-fixed-space-distribution}
Assume that $q\equiv1\pmod\ell$.  If $n$ is odd, then for every integer
$r\ge0$,
\[
\begin{aligned}
&\lim_{i\to\infty}
\frac{\#\left\{u\in\mathfrak M(\mathbb F_{q^i})\ \middle|\
\dim_{\mathbb F_\ell}
H^n(\mathcal X_{\bar u},\mathbb F_\ell)_{(1)}^{
\operatorname{Fr}_{u,i}=1}=r\right\}}
{\#\mathfrak M(\mathbb F_{q^i})}\\
&\hspace{4em}=
\frac{\#\left\{x\in\Gamma\ \middle|\
\dim_{\mathbb F_\ell}\ker(x-\operatorname{id})=r\right\}}
{\#\Gamma}.
\end{aligned}
\]
If $n$ is even, the same identity holds with the limit taken
through even integers $i$.  Here $\Gamma$ is the geometric mod-$\ell$
monodromy group of \Cref{thm:main-monodromy}.
\end{proposition}

\begin{proof}
Choose a geometric base point $s$, put $V=(\mathcal V_\ell)_s$, and let
$\rho_{\mathrm{ar}}$ be the arithmetic representation on $V$.  Denote its
arithmetic and geometric images by $\Gamma^{\mathrm{ar}}$ and $\Gamma$,
respectively.  Since the pairing takes values in $\mathbb F_\ell(-n)$,
its multiplier is $\chi_\ell^{-n}$.  The assumption
$q\equiv1\pmod\ell$ makes the cyclotomic character trivial, so
$\Gamma^{\mathrm{ar}}\subseteq\operatorname{Aut}(V,Q)$.

Write $\operatorname{pr}\colon\Gamma^{\mathrm{ar}}\to
\Gamma^{\mathrm{ar}}/\Gamma$ for the quotient map.
Identify $\operatorname{Gal}(\overline{\mathbb F}_q/\mathbb F_q)$ with
$\widehat{\mathbb Z}$ by sending geometric Frobenius to $1$.  The fundamental
exact sequence gives
\[
\begin{tikzcd}
1 \arrow[r]
& {\pi_1(\mathfrak M_{\overline{\mathbb F}_q},s)}
  \arrow[r] \arrow[d,"\rho"']
& {\pi_1(\mathfrak M,s)}
  \arrow[r,"{\deg}"] \arrow[d,"{\rho_{\mathrm{ar}}}"]
& {\widehat{\mathbb Z}} \arrow[r] \arrow[d]
& 1\\
1 \arrow[r]
& \Gamma \arrow[r]
& {\Gamma^{\mathrm{ar}}} \arrow[r,"{\operatorname{pr}}"]
& {\Gamma^{\mathrm{ar}}/\Gamma} \arrow[r]
& 1.
\end{tikzcd}
\]
Let $\bar\gamma_q$ be the image of geometric Frobenius in the finite cyclic
group $\Gamma^{\mathrm{ar}}/\Gamma$.  Then
$\rho_{\mathrm{ar}}(\operatorname{Fr}_{u,i})$ maps to
$\bar\gamma_q^{\,i}$.  For $r\ge0$, set
\(C_r=\{x\in\Gamma^{\mathrm{ar}}\mid\dim\ker(x-1)=r\}\) and
\(D_{r,i}=C_r\cap\operatorname{pr}^{-1}(\bar\gamma_q^{\,i})\),
and let $\alpha_{r,i}$ denote the proportion on the left side of the asserted
identity.  Since $\mathfrak M$ is geometrically integral, $C_r$ is
conjugacy invariant, and $\rho_{\mathrm{ar}}$ is surjective onto its image,
Chavdarov's equidistribution theorem
\cite[Thm.~4.1]{chavdarov1997generic} gives a constant
$C_{r,\ell}$, independent of $i$, such that
\begin{equation}\label{eq:chavdarov-equidistribution-error}
 \left|\alpha_{r,i}-\frac{\#D_{r,i}}{\#\Gamma}\right|
 \le C_{r,\ell}q^{-i/2}.
\end{equation}

If $n$ is odd, then
$\Gamma=\operatorname{Sp}(V,Q)$ by \Cref{thm:main-monodromy}; hence
$\Gamma^{\mathrm{ar}}=\Gamma$, $\bar\gamma_q=1$, and
$D_{r,i}=\{x\in\Gamma\mid\dim\ker(x-1)=r\}$ for every $i$.
Letting $i\to\infty$ in
\eqref{eq:chavdarov-equidistribution-error} proves the first assertion.

If $n$ is even, then $\Gamma$ is an index-two subgroup of
$\operatorname O(V,Q)$ by \Cref{thm:main-monodromy}.  Thus
$|\Gamma^{\mathrm{ar}}/\Gamma|\le2$, so $\bar\gamma_q^{\,i}=1$ for even $i$
and the same description of $D_{r,i}$ applies.  Taking the limit through even
$i$ proves the second assertion.
\end{proof}

\begin{remark}
When $n=1$, the negative eigenspace is all of
$H^1(\mathcal X_{\bar u},\mathbb F_\ell)$.  The Kummer sequence gives
\[
 \operatorname{Jac}(X_u)[\ell](\overline{\mathbb F}_q)
 \simeq H^1(\mathcal X_{\bar u},\mu_\ell)
  \simeq H^1(\mathcal X_{\bar u},\mathbb F_\ell)(1).
\]
Because $q\equiv1\pmod\ell$, the Tate twist is trivial over every
$\mathbb F_{q^i}$, so it does not change the fixed-space dimension.
Hence \Cref{prop:frobenius-fixed-space-distribution} describes the limiting
distribution of the rational $\ell$-torsion of the Jacobian.  Writing
$g=(m-2)/2$, the probability that
$\operatorname{Jac}(X_u)[\ell](\mathbb F_{q^i})\simeq
(\mathbb Z/\ell\mathbb Z)^r$ tends to
\[
 \frac{\#\{x\in\operatorname{Sp}_{2g}(\mathbb F_\ell)\mid
 \dim_{\mathbb F_\ell}\ker(x-\operatorname{id})=r\}}
 {\#\operatorname{Sp}_{2g}(\mathbb F_\ell)}.
\]
This recovers the $\ell$-torsion part of Achter's geometric Cohen--Lenstra
distribution for quadratic function fields \cite{achter2006distribution};
Achter's result is stronger in the curve case, treating $\ell^e$-torsion and
all finite abelian $\ell$-groups annihilated by $\ell^e$.
\end{remark}

\bibliographystyle{monodromyamsalpha}
\bibliography{Monodromy}
\end{document}